\documentclass[12pt]{article}

\usepackage{amsmath, amsthm, amssymb, amsfonts}
\usepackage{graphicx}
\usepackage{mathtools}
\usepackage{blindtext}
\usepackage{musicography}

\newtheorem{theorem}{Theorem}[section] %change theorem for your own names \newtheorem* makes no number

\newtheorem{lemma}{Lemma}
\newtheorem{proposition}{Proposition}
\newtheorem{corollary}{Corollary}[theorem]

\title{The Plateau Problem of Michell Trusses and Orthogonality in Springs}
\author{Chengcheng Yang}
\date{Virginia Tech\\
chengchengyang24@vt.edu\\
\today}

\begin{document}

\maketitle

\section{Abstract}
Given finitely many pointed forces in the plane. Suppose that these forces sum up to zero and their net torques also sum up to zero. It is proved in \cite{G1} that there exists a system of springs whose boundary forces exactly counter-balance these pointed forces. We will generalize to higher dimensions using the Cauchy stress tensor for elastic materials. 

Given a system of springs, we can multiply the length of each spring with its corresponding spring constant and then sum these products up. The result is called the total mass of the system. We are interested in the Plateau problem of the existence of the minimal spring system given a boundary condition; moreover, research is still going on searching for an algorithm for a minimizer.  

This minimization problem was first introduced in 1904 by A. Michell \cite{M}. He showed that a minimizer could smear out. The Michell Truss became known in mechanical engineering. It raised attention in optimal design, such as minimizing costs in building bridges. In 1960s and 1970s, the problem was developed using PDE and convex analysis by introducing an equivalent dual maximization problem. \cite{G1} \cite{Hemp} \cite{HP} \cite{M} In 2008, Bouchitt\'{e}, Gangbo, and Sppecher introduced lines of principal actions to generalize Hencky-Prandtle net to higher dimensional duality and proved that the minimizer can be found provided that it exists.  \cite{G2} In the unpublished notes of Gangbo \cite{G1}, he also showed that if springs of the same kind are optimal.

In this paper, we are going to solve the Plateau problem using two different tools in GMT: first, a minimizer can be viewed as a flat chain complex\cite{W}; second, a minimizer can also be viewed as a current \cite{F}. In Chen's thesis \cite{C}, she solved the Plateau problem for one dimension using flat chains with vector coefficients and varifolds with a new notion of signed first variation. We will extend the flat chains with matrix coefficients to all dimensions.  

At the end, we are going to show one progress in discovering the topological properties of minimizers: compressed and stretched springs must be perpendicular to each other at non-boundary points.

I appreciate my advisor Prof. Robert Hardt for communicating with me regularly on this problem.

\section{Introduction}
%Gangbo and etc. See research statement
In classic mechanics, suppose there are forces $\mathbf{F}_i$ applied at single points $M_i$ denoted as
\begin{equation*}
\mathbf{F} = \displaystyle \sum_{i=1}^l \mathbf{F}_i \delta_{M_i}, 
\end{equation*}
where $\delta_{M_i}$ is the dirac mass at $M_i$. 
If $\mathbf{F}$ is a system in equilibrium, namely the net force and the net torque are both equal to zero:
\begin{equation*}
\displaystyle\sum_{i=1}^l \mathbf{F}_i = 0, \ \ \displaystyle\sum_{i=1}^l \mathbf{F}_i \wedge M_i = 0, 
\end{equation*}
one can decompose the system into finitely many 1-dimensional beams as follows \cite{G1}:
\begin{equation*}
\mathbf{F} = \displaystyle\sum_{i, j=1}^{l'} \lambda_{ij} Beam([a_i, a_j]), 
\end{equation*}
where 
\begin{equation*}
Beam([a_i, a_j]) = (\delta_{a_j} - \delta_{a_i}) \frac{a_j - a_i}{\vert a_j -a_i \vert},
\end{equation*}
and $(\lambda_{i, j})$ is a symmetric matrix. Here $\lambda_{i, j}$ is the stress coefficient associated with the Beam$[a_i, a_j]$.

Now we generalize that to consider external forces which are uniformly distributed over $(k-1)$-simplices $[a_0^i, \ldots, a_{k-1}^i]$ denoted as
\begin{equation}\label{ycc:asystemofforcesk-1}
\mathbf{F} = \displaystyle \sum_{i=1}^l \mathbf{F}_i \mathcal{H}^{k-1}_{\vert [a_0^i, \ldots, a_{k-1}^i]},
\end{equation}
where $\mathcal{H}^{k-1}_{\vert [a_0^i, \ldots, a_{k-1}^i]}$ is the $(k-1)$-dimensional Hausdorff measure in $\mathbb{R}^n$ restricted to the simplex $[a_0^i, \ldots,  a_{k-1}^i]$; and $\mathbf{F}_i$ is the (constant) force exerted on $[a_0^i, \ldots,  a_{k-1}^i]$. 
If $n \geq k+1$ and the same equilibrium conditions are satisfied, then we can show that $\mathbf{F}$ can be decomposed into finitely many $k$-dimensional simplices with elasticities known as Cauchy stress tensors. In other words, there exist a $k$-dimensional elastic simplicial complex whose elastic forces on the boundary exactly counterbalance the system $\mathbf{F}$ provided that the support of $\mathbf{F}$ consists of finitely many $(k-1)$-simplices. The proof is technical and we employed tools such as differential forms, Hodge star operator, left interior multiplication. 

Physically, one can image finitely many surfaces under external forces, and the question is whether there exist elastic bodies whose boundary forces match those on the surfaces. 

%an elastic membrane or a soap film is partially bounded by finitely many piecewise linear curves and supports heavy external forces. Next instead of any elastic membrane, we consider these consisting of finitely many polygonal surfaces. The shape of the complex, together with the stress and strain, will suggests us to ask the following two questions:
%\begin{itemize}
%\item[(i) ] What does one with the minimum elastic energy look like?  
%\item[(ii)] What is an optimal design so that the total cost is a minimum? 
%\end{itemize}

%For these questions, the case $k=1$ was first introduced by A. Michel in 1904, and then used in mechanical engineering, and then recently popularized in many pure mathematics works by Gangbo and others. We are investigating the general case $k \geq 2$. For example, we already have formulae for the elastic energy and the cost, and we could also control them within a certain bound. But the harder question is how to pin down the minimizers. Here geometric measure theory is applied. 

\subsection{Rank 1 Diagonal Matrices}
Let $D$ be an $n \times n$ diagonal matrix, then $D$ has rank 1 if and only if $D = \lambda e_i \otimes e_i$ for some nonzero real number $\lambda$ and $i \in \{1, \ldots, n\}$. Here we use the notation that given any vector $v = (v_1, \ldots, v_n)$, $v \otimes v$ denotes the $n \times n$ matrix whose $(i, j)$-entry is equal to $v_iv_j$. 

In general, let $A$ be an $n \times n$ symmetric matrix, then $A$ has rank 1 if and only if $A = \lambda v \otimes v$ for some unit vector $v \in \mathbb{R}^n$ and some nonzero real number $\lambda$. We may choose an orthogonal matrix $O$ such that $D = O^TA O$ is diagonal. Then rank($A$) = rank($D$) = 1.

\subsection{Polyhedral $k$-chains in $\mathbb{R}^n$}
For integers $0 \leq k \leq n$ and geometrically independent points $a_0, a_1, \ldots, a_k$, the $k$-simplex spanned by $a_0, a_1, \ldots, a_k$ is the set (convex hull) of all points $x$ of $\mathbb{R}^n$ such that 
\[x = \displaystyle \sum_{i = 0}^k t_i a_i, \text{ where } \displaystyle \sum_{i=0}^k t_i = 1, \text{ and }t_i \geq 0. \]
Since there are two orientations, we use $[a_0, a_1, \ldots, a_k]$ to denote the oriented $k$-simplex consisting of the simplex $a_0, a_1, \ldots, a_k$ and the equivalence class of the particular ordering $(a_0, a_1, \ldots, a_k)$.

Let $A_k$ denote the collection of all oriented $k$-simplices in $\mathbb{R}^n$. Then we define $C_k$ to be the real vector space generated by $A_k$, mod out the subspace that is generated by elements of the form 
\[\sigma + \hat{\sigma}, \text{ and } \sigma_1 + \sigma_2 - \sigma_1 \cup \sigma_2,\]
where $\hat{\sigma}$ has the opposite orientation of $\sigma$. And $\sigma_1$ and $\sigma_2$ share a common $(k-1)$-dimensional face and their union is another $k$-simplex. 

For $k \geq 1$, the boundary operator $\partial$ is defined as usual: 
\[\partial [a_0, \ldots, a_k] = \displaystyle \sum_{i=0}^k (-1)^k [a_0, \ldots, \hat{a_i}, \ldots, a_k],\] 
where $\hat{a_i}$ means omitting $a_i$. One needs to check the well-definedness on $C_k$. 
When $k=1$, suppose an oriented 1-simplex $[a, b]$ is being subdivided by an interior point $c$, we orient the two new $1$-simplices as $[a, c]$ and $[c, b]$ so that 
\[\partial ([a, c] + [c, b] - [a, b]) = 0. \]
In general, suppose we've have done so for $C_{k-1}$ such that $\partial$ is well-defined. Then given a $k$-simplex $\sigma$ being subdivided into two simplices $\sigma_1$, $\sigma_2$ of dimension $k$, there is a vertex $w$ and a face $s$ such that $\sigma$ is the cone $[w, s]$ of $w$ over $s$, together with $s$ being subdivided into two simplices of dimension $k-1$, calling them $s_1$, $s_2$. It follows that $\sigma_i$ is the cone $[w, s_i]$ for $i=1,2$. Note that if $s=[a_0, \ldots, a_{k-1}]$ is an oriented simplex in $C_{k-1}$, the bracket $[w, s]$ denotes the oriented simplex $[w, a_0, \ldots, a_{k-1}]$. 
Then one may check 
\begin{eqnarray*}
&& \partial \sigma_1 + \partial \sigma_2 - \partial \sigma \\
&=& \partial([w, s_1] + [w, s_2] - [w, s]) \\
&=& (s_1 - [w, \partial s_1]) + (s_2 - [w, \partial s_2]) - (s - [w, s])\\
&=& (s_1 + s_2 - s) - [w, \partial (s_1 + s_2 - s)]\\
&=& 0.
\end{eqnarray*}
The last step is by the inductive definition. Therefore {\bf a polyhedral $k$-chain} in $C_k$ is a finite real linear combination of (possibly overlapping) oriented $k$-simplices. It is equivalent to a finite real linear combination of disjoint (except on their boundaries) oriented $k$-simplices. Moreover, applying the boundary operator $\partial$ to both yields the same result. 
Given that $\partial^2=0$, we obtain a chain of real vector spaces: 
\[\{C_k, \partial\} = C_n \xrightarrow{\partial} \cdots \xrightarrow{\partial} C_0 \rightarrow 0,\] 
when the ambient space is $\mathbb{R}^n$. 
%We call such a chain complex $\mathcal{C} =\{C_k, \partial\}$ a {\bf polyhedral} chain complex in $\mathbb{R}^n$. 

Imitating the de Rham theorem, one can easily define the integral 
\[\int_{\sigma} \omega = \int_{\Delta_k} \sigma^{\ast}\omega, \text{ where } \sigma: \Delta_k = [e_0, \ldots, e_k] \rightarrow [a_0, \ldots, a_k], \]
for any smooth differential $k$-form $\omega$ and any oriented $k$-simplex $\sigma = [a_0, \dots, a_k]$. Extending by linearity, for any $k$-chain $c$ one can define
\[\int_{c} \omega = \sum_{i = 1}^m c_i \int_{\sigma_i} \omega, \, \text{ if } c = \sum_{i=1}^m c_i \sigma_i.\]

\subsection{Stressed $k$-chains in $\mathbb{R}^n$}
In Solid Mechanics, the {\it Cauchy Stress Tensor} is a $3 \times 3$ symmetric matrix $A$ that tells you the force of one part of the body acting on the other. More precisely, each point $a$ of an elastic body in $\mathbb{R}^3$ is associated with an $A$  such that if a plane passing through $a$ with a unit normal $v$, the force of the half space $\{x \in \mathbb{R}^3: (x - a) \cdot v < 0\}$ acting on the plane per unit area is given by $Av$. Here we are going to use the same definition with $\mathbb{R}^3$ replaced by $\mathbb{R}^n$ for any integer $n \geq 1$. 

A (constantly) {\it stressed polyhedral $k$-simplex} is a tensor product $A \otimes \sigma$ where $\sigma$ is an oriented $k$-simplex and $A$ is a symmetric $n \times n$ real matrix. Therefore a {\it stressed polyhedral $k$-chain} is a finite real linear combination of stressed polyhedral $k$-simplices. 

\subsection{One Dimensional Example: Stressed Springs.} Suppose $n=2$, $k=1$, and $a = (0, 0)$, $b=(1, 0)$ are two points on the horizontal axis. Let $\sigma$ be the oriented 1-simplex $[a, b]$ and $A = D_{\mu e_1} =\mu(e_1 \otimes e_1)$. 
%So $A \otimes \sigma$ is a structrually stressed 1-simplex, since the nonzero eigenvector of $A$ is $e_1$. 

For $\mu > 0$, $A \otimes \sigma$ acts like a {\it compressed spring} because the external force in the positive $e_1$ direction is a positive multiple of $e_1$. A compressed spring means that it pushes outward on its endpoints. 

Similarily, for $\mu < 0$, $A \otimes \sigma$ acts like a {\it stretched spring} because the exteranl force in the positive $e_1$ direction is a negative multiple of $e_1$. A stretched spring means that it pulls inward on its endpoints. 

A general {\it structurally stressed spring} between two arbitrary points $a, b \in \mathbb{R}^n$ is 
described by a {\it structurally stressed} 1-simplex in the form of $D_{\mu(b-a)} \otimes [a, b]$, where $D_{\mu(b-a)}$ describes the coefficient of $[a, b]$. Note that the eigenvalue vector of $D_{\mu(b-a)}$ lies in the linear space spanned by $(b-a)$.

Michel strusses are finite linear combinations of structurally stressed springs that efficiently balance a given system of pointed vector forces. 

Note that a simple {\it nonstructurally stressed} 1-simplex is $D_{e_2} \otimes [a, b]$ with $a = (0, 0), b=(1, 0)$. 

\subsection{Higher Dimensional Definitions.} We will define  a stressed $k$-simplex $A \otimes \sigma$ to be {\it structurally stressed} if all the eigenvectors of $A$ lie in the tangent space of $\sigma$. 

A {\it stressed polyhedral $k$-chain} is a finite sum $\sum_{i = 1}^m \lambda_i A_i \otimes \sigma_i$ of stressed $k$-simplices. The collection of these gives a real vector space 
\[\mathcal{P}_k (\mathbb{R}^n; Sym_n) = Sym_n \otimes C_k.\]
For $k \geq 1$, the usual formula for the boundary $\partial \sigma$ of an oriented simplex $\sigma$ leads to a well-defined {\it boundary operator}
\[\partial: \mathcal{P}_k (\mathbb{R}^n; Sym_n) \rightarrow \mathcal{P}_{k-1} (\mathbb{R}^n; Sym_n), \, 
\partial (\sum_{i = 1}^m \lambda_i A_i \otimes \sigma_i) = \sum_{i = 1}^m \lambda_i A_i \otimes \partial \sigma_i\]
A polyhedral chain is {\it structurally stressed} if all its simplices are structurally stressed.

The boundary of a structurally stressed polyhedral $k$-simplex is usually {\it not} entirely structurally stressed. The situation is trivial when $k=1$ because a nonzero coefficient $A$ would contain an eigenspace of dimension $\geq 1$ which doesn't exist in the tangent space of a point. A more interesting and instructive elementary example is as follows.

\subsection{A Uni-directional Stretch of a Rectangle.}
Suppose $n=2$, $k=2$, and $S$ is a counterwise-oriented coordinate rectangle $[a, b] \times [c, d] \in \mathbb{R}^2$ for some real numbers $a < b, c< d$. Let $A$ be the simple horizontal stretching $D_{e_1} = (e_1 \otimes e_1)$. One can imagine $S$ is a horizontally stretched piece of elastic cloth.
 Then $\partial S$ is the sum of four terms corresponding to the four oriented edges of $S$:
\[\partial S = A \otimes [(a, c), (b, c)] + A \otimes [(b, c), (b, d)] + A \otimes [(b, d), (a, d)] + A \otimes [(a, d), (a, c)].\]
The first and third are top and bottom edges, being stretched horizontally. So they are structurally stressed. But the second and forth are right and left edges, being neither compressed nor stretched, but instead pulled apart horizontally. So they are nonstructurally stressed. 

Therefore in general the boundary of a structurally stressed simplex has two parts: the stress in the boundary and the external force coming from the original stressed simplex. Let's look at another example.

\subsection{A Uni-directional Stretch of a Triangle} Next let's cut the rectangle $S$ along a diagonal. Then for either of the resulting two right triangles, the diagonal boundary edge would be both stretched internally and pulled externally. First we abbreviate the three oriented edges of the lower stressed right triangle $T$ as the oriented 1-simplices:
\[I = [(a, c), (b, c)], \, J = [(a, c), (a,d)], \, K = [(a, d), (b, c)]. \]
See Figure \ref{fig:stressforce}.

\begin{figure}
    \centering
    \includegraphics[width=1\linewidth]{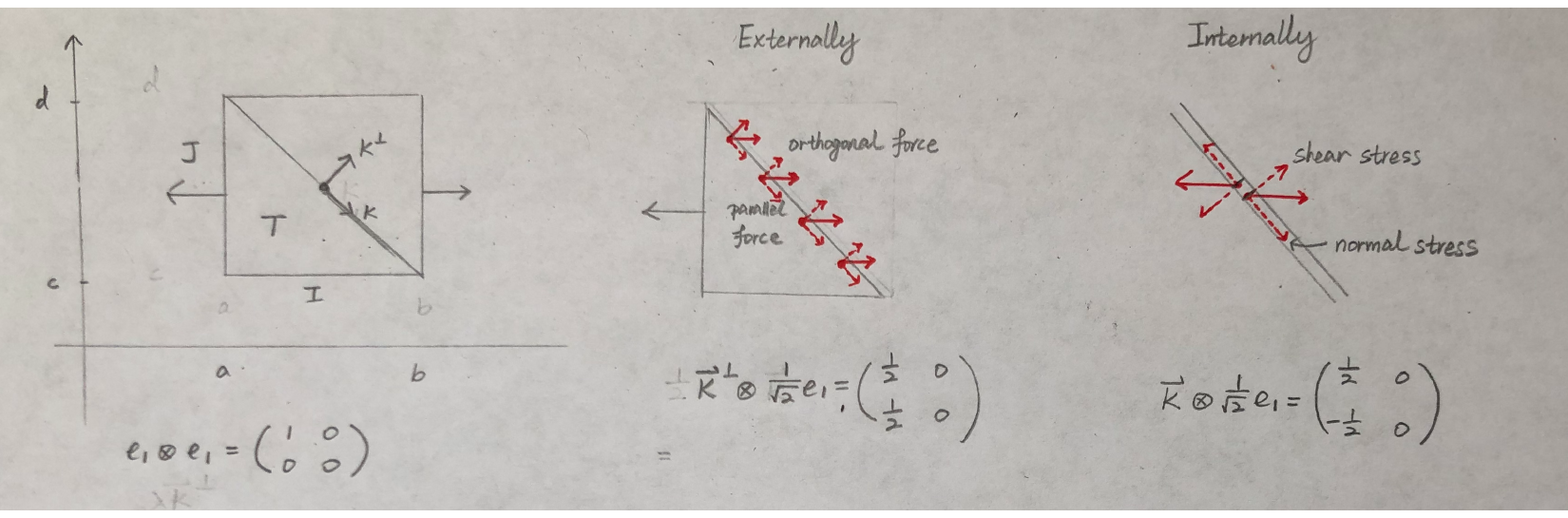}
    \caption{The decomposition of the matrix.}
    \label{fig:stressforce}
\end{figure}

To describe the effect of the horizontal stretch $A$ on the diagonal edge $K$, we simply orthogonally decompose $e_1$ in terms of its components:
\[e_1 = \frac{1}{\sqrt{2}} \vec{K} + \frac{1}{\sqrt{2}} \vec{K}^{\perp},\] 
where $\vec{K} = (\frac{1}{\sqrt{2}}, -\frac{1}{\sqrt{2}})$  is the unit vector of $K$, and $\vec{K}^{\perp} =(\frac{1}{\sqrt{2}}, \frac{1}{\sqrt{2}})$ is outward pointing unit vector perpendicular to $\vec{K}$. 

Then 
\begin{equation*}
\begin{split}
e_1\otimes e_1 &= (\frac{1}{\sqrt{2}} \vec{K} + \frac{1}{\sqrt{2}} \vec{K}^{\perp})\otimes (\frac{1}{\sqrt{2}} \vec{K} + \frac{1}{\sqrt{2}} \vec{K}^{\perp}) \\
& =  \left[1/2 (\vec{K}^{\perp} \otimes \vec{K}^{\perp}) +1/2 (\vec{K}^{\perp} \otimes \vec{K})\right] + \left[1/2 \vec{K} \otimes \vec{K} + 1/2 \vec{K} \otimes \vec{K}^{\perp}\right] \\
& = \vec{K}^{\perp}  \otimes \frac{1}{\sqrt{2}}e_1 +  \vec{K} \otimes \frac{1}{\sqrt{2}} e_1.
\end{split}
\end{equation*}

The external force on the diagonal edge $K$, coming from the lower triangle, is 
represented by the matrix: $\vec{K}^{\perp}  \otimes \frac{1}{\sqrt{2}}e_1$. If we multiply this matrix with $\vec{K}^{\perp}$, we obtain $\frac{1}{\sqrt{2}}e_1$. Since this is the force per unit length, we need to multiply it with the length of the diagonal, which is $\sqrt{2}$, and we get $e_1$ as desired. Moreover, this external force has two components: one is orthogonal to $K$ and the other is parallel to $K$ as expressed into the following two matrices:
\[1/2 (\vec{K}^{\perp} \otimes \vec{K}^{\perp}), 1/2 (\vec{K}^{\perp} \otimes \vec{K}).\]

On the other hand, at each interior point of the diagonal edge $K$, there is a pair of opposite horizontal forces $\{e_1, -e_1\}$, which contributes to the stress inside the edge. It is represented by the matrix: $\vec{K} \otimes \frac{1}{\sqrt{2}} e_1$. We can also decompose the stress. The pair of forces parallel to the diagonal gives us the normal stress, just like a spring, and it is given by the matrix: $1/2 \vec{K} \otimes \vec{K}$; the pair of forces orthogonal to the diagonal gives us the shear stress, unlike a spring but tearing the spring apart, and it is given by the matrix: $1/2 \vec{K} \otimes \vec{K}^{\perp}$.

\subsection{Multi-directional Stretches of a Rectangle.}
By the spectral theorem one can similarly understand the higher rank case by writing any symmetric matrix $A$ as 
\[A = \displaystyle \sum_{i=1}^n \mu_i v_i \otimes v_i= \displaystyle \sum_{i=1}^n D_{\mu_iv_i}, \]
where $\{v_i\}$ is an orthonormal basis consisting of eigenvectors of $A$, and the $\mu_i$ are their corresponding eigenvalues. 

As a specific example of this, suppose that one works with the same rectangle $S$ as above but replaces $D_{e_1}$ by $ A = D_{e_1} + D_{e_2}$. Intuitively, imagine that an elastic cloth is made of both horizontal and vertical threads with possibly different elasticity. Then the resulting $A$ stressed rectangle is now being {\it simultaneously} stretched both horizontally and vertically. The top and bottom edges are simultaneously stretched horizontally and pulled apart vertically, while the right and left edges are simultaneously stretched vertically and pulled horizontally.

\subsection{Decomposition Theorem}
Turning now to $(k-1)$-dimensional stressed simplex $A \otimes \tau$ occurring in the boundary expansion of a general structurally stressed $k$-simplex $A \otimes \sigma$ in $\mathbb{R}^n$. It suffices to work on the rank 1 case, say $A = \mu e_1 \otimes e_1$. 
In the $k$-dimensional tangent space of $\sigma$, let $\vec{\tau}^{\perp}$ be the outward normal unit vector to the tangent space of $\tau$. As before, we find the orthogonal decomposition of $e_1$ in terms of $\vec{\tau}^{\perp}$ and its projection into the tangent space of $\tau$.

Notice that when we decompose the symmetric matrix $A$ into two submatrices, neither of the two matrices is symmetric. Therefore, the general vector space $M_n$ of all $n \times n$ real matrices is considered. We can now state the following theorem.

\begin{theorem}\label{ycc:thm1} 
For any structurally stressed polyhedral $k$-chain P in $\mathbb{R}^n$, the $(k-1)$-stressed chain $\partial P$ admits a unique decomposition:
\begin{equation}\label{ycc:thm1:eqn1}
\partial P = S + F, 
\end{equation}
where $S$ consists of simplices in $M_n \otimes C_{k-1}$ with only normal and shear stresses and $F$ consists of simplices in $M_n \otimes C_{k-1}$   with only externally orthogonal and parallel forces. 
\end{theorem}

\begin{proof}Given any structurally stressed polyhedral $k$-chain $P$, it is a finite sum of structurally stressed polyhedral $k$-simplex $A \otimes \sigma$, where $\sigma$ is an oriented $k$-simplex and $A \in Sym_{n}$. The spectral theorem assures that $A$ is diagonalizable and so 
\[A = D_{\mu_1 v_1} + \cdots + D_{\mu_n v_n} = \mu_1 (v_1 \otimes v_1) + \cdots + \mu_n (v_n \otimes v_n),\]
where $v_1, \ldots, v_n$ are orthonormal vectors.
By hypothesis, one may assume, without loss of generality, that $v_l$ lies in the tangent space of $\sigma$ for $l \leq k$ and $\mu_l = 0$ for $l > k$. 

Suppose $\sigma = [a_0, \ldots, a_k]$, then span($a_1-a_0, \ldots, a_k-a_0$) is the tangent space of $\sigma$ in $\mathbb{R}^n$. Given that
\[\partial \sigma = \displaystyle \sum_{i=0}^k (-1)^i [a_0, \dots, \hat{a_i}, \dots, a_k],\]
let $s = [a_0, \ldots, \hat{a_i}, \dots, a_{k}]$ be the face opposite to the vertex $a_i$, then the set $\{a_1 - a_0, \ldots, \widehat{a_i - a_0}, \dots, a_k-a_0\}$ forms a basis for the tangent space of $s$; moreover, there is a unit normal vector $\hat{n}$ to the space such that the basis $\{a_1-a_0, \ldots, \widehat{a_i-a_0}, \ldots, a_{k}-a_0, \hat{n}\}$ spans the same space as $\{a_1-a_0, \ldots, a_k-a_0\}$. We may choose $\hat{n}$ to point outward, so that the two bases are consistently oriented. In other words, if the transition matrix $B$ is defined by
\[(-1)^i(a_1-a_0, \ldots, \widehat{a_i-a_0}, \ldots, a_{k}-a_0, \hat{n}) = B(a_1-a_0, \ldots, a_k-a_0),\]
then $B$ has a positive determinant. 
%So we get $\hat{n}$ based on the right-hand rule.

Next following the Gram-Schmidt algorithm there is an orthonormal basis $\{\vec{E}_1, \ldots, \vec{E}_{k-1}\}$ for the tangent space of $s$ which satisfy the following relations: 
\begin{equation*}
\vec{E}_j  = 
\begin{cases} 
\displaystyle\frac{(a_j-a_0) - \sum_{l=1}^{j-1} \langle a_j-a_0, \vec{E}_l \rangle \vec{E}_l}{\vert (a_j-a_0) - \sum_{l=1}^{j-1} \langle a_j-a_0, \vec{E}_l \rangle \vec{E}_l \vert} \ \  \text{ for $j \leq i-1$}, \\
\\
\displaystyle\frac{(a_{j+1}-a_0) - \sum_{l=1}^{j-1} \langle a_{j+1}-a_0, \vec{E}_l \rangle \vec{E}_l}{\vert (a_{j+1}-a_0) - \sum_{l=1}^{j-1} \langle a_{j+1}-a_0, \vec{E}_l \rangle \vec{E}_l \vert} \ \  \text{ for $j\geq i$},
\end{cases}
\end{equation*}
where $\vec{E}_1 = (a_1-a_0)/\vert a_1 - a_0 \vert$. Note that the orientation of the orthonormal basis $\{\vec{E}_1, \ldots, \vec{E}_{k-1}, \hat{n}\}$ is still consistent with the orientation before, because the Gram-Schmidt algorithm was performed under a sequence of elementary matrices with positive determinants. 

Therefore, we may define an orthogonal matrix $O$ as the change of coordinates matrix whose first $k$ columns are
 $O = (\vec{E}_1, \ldots, \vec{E}_{k-1}, \hat{n},  \ldots)$. 
For each $l\leq k$, write $v_l$ as follows:
\begin{equation*}
v_l = \kappa_1^l \vec{E}_1 + \ldots + \kappa_{k-1}^l \vec{E}_{k-1} + \kappa_k^l \hat{n}.
\end{equation*}
Then the matrix $OD_{\mu_lv_l}O^T$ is zero everywhere except the upper left block matrix
\begin{equation*}
\mu_l \left[\begin{matrix} (\kappa_1^l)^2 & \cdots & \kappa_1^l \kappa_k^l \\
\vdots & \ddots & \vdots \\ \kappa_k^l \kappa_1 & \cdots & (\kappa_k^l)^2 \end{matrix} \right] = \mu_l (\kappa_p^l \kappa_q^l).
\end{equation*}
Combining everything together, we have
\begin{equation*}
A  = O\left[ \displaystyle \sum_{l=1}^k \sum_{p, q = 1}^k \mu_l (\kappa_p^l \kappa_q^l)\right]O^T.
\end{equation*}
Let's look at the matrix in the brackets in details. Denote it as $\tilde{A}$, then it can be separated into four matrices as follows:
\begin{eqnarray}\label{ycc:tildeA}
\tilde{A} &=& \left[\begin{matrix} \tilde{a}_{1,1} & \cdots & \tilde{a}_{1, k-1} & 0 \\ 
\vdots & \ddots & \vdots &  \vdots \\ \tilde{a}_{k-1, 1} & \cdots & \tilde{a}_{k-1, k-1} & 0 \\ 
0 & \cdots & 0 & 0 \end{matrix}\right] + 
\left[\begin{matrix} 0 & \cdots & 0 & 0 \\ 
\vdots & \ddots & \vdots &  \vdots \\ 0& \cdots & 0 & 0 \\ 
\tilde{a}_{k, 1} & \cdots & \tilde{a}_{k, k-1} & 0 \end{matrix}\right] \\  \notag
&& + \left[\begin{matrix} 0 & \cdots & 0 & 0 \\ 
\vdots & \ddots & \vdots &  \vdots \\ 0 & \cdots & 0 & 0 \\ 
0 & \cdots & 0 & \tilde{a}_{k, k} \end{matrix}\right] + \left[\begin{matrix} 0 & \cdots & 0 & \tilde{a}_{1, k} \\ 
\vdots & \ddots & \vdots &  \vdots \\ 0 & \cdots & 0 & \tilde{a}_{k-1, k} \\ 
0 & \cdots & 0 & 0 \end{matrix}\right]. 
\end{eqnarray}

The interpretations of the four matrices are analogous to a uni-stretch of a triangle. 
\begin{enumerate}
\item The first symmetric matrix is the usual Cauchy stress tensor inside the boundary face $s$. We may view $s$ as an elastic surface (with a fixed thickness) of dimension $k-1$ in the tangent space spanned by $s$, calling it $\mathbb{R}^{k-1}$. Then at each interior point $a$, if $\vec{u} \in \mathbb{R}^{k-1}$ is a unit vector, then multiplying $\vec{u}$ with the first matrix gives the stress (force per unit area) at $a$ in the orthogonal plane $\{x \in \mathbb{R}^{k-1}: (x-a) \cdot \vec{u} = 0\}$ exerted by the half space $\{x \in \mathbb{R}^{k-1}: (x-a) \cdot \vec{u} < 0\}$. For instance $\vec{u} = \vec{E}_1$, then the stress is equal to (up to thickness of $s$)
\begin{equation*}
\left[\begin{matrix} \tilde{a}_{1,1} & \cdots & \tilde{a}_{1, k-1} & 0 \\
\vdots & \ddots & \vdots &  \vdots \\ \tilde{a}_{k-1, 1} & \cdots & \tilde{a}_{k-1, k-1} & 0 \\
0 & \cdots & 0 & 0 \end{matrix}\right] \vec{u} = \tilde{a}_{1, 1}\vec{E}_1 + \cdots + \tilde{a}_{k-1, 1}\vec{E}_{k-1}.
\end{equation*}
Since the stress is contained inside the tangent space, we will use the same definition as in $\mathbb{R}^2$ or $\mathbb{R}^3$ to call this {\it normal stress}.

\item The second matrix in the first row is the shear stress for face $s$. With the same setup as above,  if we multiply the unit vector $\vec{u}$ with the second matrix, the result is a perpendicular vector to the face. Hence, if we view $s$ as an elastic body, then the {\it shear stress} at $a$ on the orthogonal plane $\{x \in \mathbb{R}^{k-1}: (x-a) \cdot \vec{u} = 0\}$ due to $\{x \in \mathbb{R}^{k-1}: (x-a) \cdot \vec{u} < 0\}$ is equal to the product. For example if $\vec{u} = u_1\vec{E}_1 + \cdots + u_{k-1}\vec{E}_{k-1}$, the shear stress is equal to (up to thickness of $s$)
\begin{equation*}
\left[\begin{matrix} 0 & \cdots & 0 & 0 \\
\vdots & \ddots & \vdots &  \vdots \\ 0& \cdots & 0 & 0 \\
\tilde{a}_{k, 1} & \cdots & \tilde{a}_{k, k-1} & 0 \end{matrix}\right]\vec{u} = (\tilde{a}_{k, 1}u_1+ \cdots + \tilde{a}_{k, k-1}u_{k-1})\hat{n}.
\end{equation*}

\item The third matrix in the second row tells us that the original $k$-simplex $\sigma$ is pushing or pulling its face $s$ in the orthogonal direction. This external {\it orthogonal force} $\vec{F}_1$ can be calculated using 
\begin{eqnarray*}
\vec{F}_1 &=& - \tilde{a}_{k, k} (\text{surface area of $s$})\hat{n} \\
 &=& -\tilde{a}_{k, k} \sqrt{\text{Gram}(a_1-a_0, \ldots, \widehat{a_i - a_0}, \ldots, a_k-a_0)} \  \hat{n}.
\end{eqnarray*}

\item The fourth matrix indicates that the $k$-simplex $\sigma$ exerts an external force on the face $s$ parallel to the tangent space of $s$. And this external {\it parallel force} $\vec{F}_2$ is equal to 
\begin{eqnarray*}
\vec{F}_2 &=& -(\tilde{a}_{1, k}\vec{E}_1 + \cdots + \tilde{a}_{k-1, k} \vec{E}_{k-1})(\text{surface area of $s$})\\
&=& -(\tilde{a}_{1, k}\vec{E}_1 + \cdots + \tilde{a}_{k-1, k} \vec{E}_{k-1})\sqrt{\text{Gram}(a_1-a_0, \ldots, \widehat{a_i - a_0}, \ldots, a_k-a_0)}.
\end{eqnarray*}
\end{enumerate}

Denote the four matrices above in $\tilde{A}$ as $\tilde{A}_1, \tilde{A}_2, \tilde{A}_3, \tilde{A}_4$, respectively. Then we may write $A \otimes s$ as 
\begin{eqnarray*}
A \otimes s  &=& O \tilde{A}_1O^T  \otimes s + O \tilde{A}_2O^T  \otimes s + O \tilde{A}_3O^T  \otimes s + O \tilde{A}_4O^T  \otimes s \\
&=& O(\tilde{A}_1 + \tilde{A}_2)O^T \otimes s+ O(\tilde{A}_3 + \tilde{A}_4)O^T \otimes s.
\end{eqnarray*} 
The first term gives only information about the internal stresses within the face $s$ and the second term gives only information about the external forces on the face $s$ from the $k$-simplex $\sigma$,
each of which corresponds to the $S$ and $F$ in (\ref{ycc:thm1:eqn1}), thus finishing the proof of existence. The uniqueness is easy to see.
\end{proof}

\subsection{Generalized Cauchy Stress Tensor}
The coefficient $n \times n$ symmetric matrix for lower dimensional $k$-chains, $k < n$, carries not only information about stresses (compress, stretch, shearing), but also information about external forces (being pulled or pushed). It extends the usual Cauchy stress tensor. The theorem leads to the following definition of a generalized Cauchy stress tensor. 

Let a stressed polyhedral $(k-1)$-simplex $A \otimes \tau$ be given where $A$ is an $n \times n$ real  symmetric matrix and $\tau$ is an oriented $(k-1)$-simplex. Suppose that the orthonormal eigenvectors $\vec{v}_i$ and their corresponding eigenvalues $\mu_i$ of $A$ satisfy the two conditions:
\begin{itemize}
\item[(i)] One can reorder the $\vec{v}_i$ so that $\mu_i \neq 0$ \text{ for } $i \leq k$ and $\mu_i = 0 \text{ for } i > k$;
\item[(ii)] Moreover, the first $k$ eigenvectors $\vec{v}_1, \ldots, \vec{v}_{k}$ span a $k$-dimensional subspace that contains the tangent space of $\tau$. 
\end{itemize}
Then we say that $A$ is a {\it generalized Cauchy stress tensor} for $\tau$. In the special case that $\vec{v}_1, \ldots, \vec{v}_{k-1}$ span the tangent space $\tau$, $\sum_{i=1}^{k-1} D_{\mu_iv_i} \otimes \tau$ represents the normal stress in $\tau$ and $D_{\mu_kv_k}$  represents the orthogonal forces on $\tau$.

\section{The Converse to Decomposition Theorem}

The next natural question is whether the converse of the theorem also holds. In other words, suppose $Q$ is a stressed polyhedral $(k-1)$-chain, we look for necessary and sufficient conditions for which there exists a structurally stressed polyhedral $k$-chain $P$ such that $\partial P = Q$. 

Physically, one can imagine $Q$ as finitely many 
$(k-1)$-dimensional polyhedrons made of elastic materials in $\mathbb{R}^n$.  Suppose that each polyhedron in $Q$ also is under some pressure, then we ask the question: ``Does there exist a finite number of elastic $k$-dimensional polyhedrons $P$ such that 
\begin{itemize}
\item[(ii)] The boundary forces of $P$ match the external forces on $Q$;
\item[(iii)] The stresses on the boundary of $P$  also match those in $Q$?"
\end{itemize}

The answer is yes! In the next two subsections, we will prove this result.

\subsection{Balancing External Forces}
Given a system $\mathbf{F}$ of forces in $\mathbb{R}^n$ whose support consists of finitely many $(k-1)$-simplices. If $n \geq k+1$, we are going to show that if the system $\mathbf{F}$ is in equilibrium,
there exists a $k$-dimensional stressed chain such that its external forces on the boundary exactly balance the system $\mathbf{F}$. The proof is technical. It consists of three lemmas for radial cases and three propositions for non-radial cases.

\subsubsection{Set Up}
Recall that if a system $\mathbf{F} = \sum_{i=1}^l \mathbf{F}_i \delta_{M_i}$ of pointed forces is in equilibrium, there is a finite number of springs $\sum_{i, j=1}^{l'} \lambda_{ij} [a_i, a_j]$ such that their boundary forces match $\mathbf{F}$. More precisely, one defines
\[Beam([a_i, a_j]) = (\delta_{a_j} - \delta_{a_i}) \frac{a_j - a_i}{\vert a_j -a_i \vert},\]
then $\lambda_{i, j} > 0$ for stretched springs and $\lambda_{i, j} < 0$ for compressed ones. It follows that
\begin{equation*}
\mathbf{F} = \displaystyle\sum_{i, j=1}^{l'} \lambda_{ij} Beam([a_i, a_j]).
\end{equation*}

In general the equilibrium conditions for a system of forces in the form of (\ref{ycc:asystemofforcesk-1}) can be written as:
\begin{equation}\label{ycc:equilibriumk-1}
\displaystyle\sum_{i=1}^l \mathbf{F}_i = 0, \ \ \displaystyle\sum_{i=1}^l \mathbf{F}_i \wedge \hat{s_i}=0, 
\end{equation}
where $\hat{s_i} = (a_0^i + \cdots + a_{k-1}^i)/k$ is the barycenter of the simplex $s_i = [a_0^i , \ldots , a_{k-1}^i]$.

When we generalize the Beam to higher dimension, one has to consider orientation. Vectors in $\mathbb{R}^n$ can be compared with each other lexicographically. For example, 
\[(1, 2, 5) < (1, 3, -4),\]
because the second coordinates satisfy $2< 3$. Thus given a $k$-simplex $[a_0, \ldots, a_k]$ in $\mathbb{R}^n$, the orientation is chosen so that
\[a_0 < a_1 < \cdots < a_k.\]
Suppose its Cauchy stress tensor $A$ is an $n \times n$ matrix whose nonzero eigenvalues are $\mu_1, \ldots, \mu_k$ corresponding to orthonormal eigenvectors $v_1, \ldots, v_k$. We assume that $v_i > 0$ for each $1 \leq i \leq k$. 

\subsubsection{Boundary Forces}
We are going to show that the forces on the boundary of $[a_0, a_1, \ldots, a_k]$ due to the Cauchy stress tensor $A = D_{\mu_1v_1}+\cdots D_{\mu_kv_k}$ is equal to

\begin{eqnarray}\label{ycc:kbeam}
& & A \otimes Beam[a_0, \ldots, a_k] \nonumber \\ 
&=&\frac{1}{(\text{volume of $\sigma$})}  \displaystyle \sum_{i=0}^k (-1)^i \displaystyle \sum_{j=1}^k \mu_j v_j \left[ (\zeta \ \llcorner \ d\eta_i) \bullet v_j \right] \ d\mathcal{H}^{k-1}_{\vert [a_0, \ldots, \hat{a_i}, \ldots, a_k]}. 
\end{eqnarray}
Here $\llcorner: \bigwedge_k \mathbb{R}^n \times \bigwedge^{k-1} \mathbb{R}^n \rightarrow \bigwedge_1 \mathbb{R}^n$ is the left interior multiplication with $\zeta$ and $d\eta_i$ defined as follows:
\begin{eqnarray*}
\zeta &=& (a_1 - a_0) \wedge \cdots \wedge (a_k-a_0), \\
d\eta_i &=& d(a_1-a_0) \wedge  \cdots \wedge \widehat{d(a_i - a_0)} \wedge \cdots \wedge d(a_k - a_0) \text{ for } i > 0,  \\
d\eta_0 &=& d(a_2-a_1) \wedge \cdots \wedge d(a_k-a_1),
\end{eqnarray*}
where $d(w)$ represents the dual vector to $w \in \mathbb{R}^n$ in the dual space $(\mathbb{R}^{n})^{\ast}$. 
Note that when $n=k$, the left interior multiplication $\llcorner$ is the same as the Hodge star $\ast$.

\begin{lemma}
Given the above notations, first $(-1)^i(\zeta \  \llcorner  \ d\eta_i)$ gives the correct normal vector to the $i$th face $s_i = [a_0, \ldots, \hat{a_i}, \ldots, a_k]$ in $\mathbb{R}^n$.  Second
\[\frac{\mu_j v_j [(-1)^i(\zeta \  \llcorner  \ d\eta_i)] \bullet v_j}{\text{volume of $\sigma$}}\]
 is  equal to the force exerting on $s_i$ due to the stress in the direction of the $j$th eigenvector $v_j$ with eigenvalue $\mu_j$, i.e., the matrix $D_{\mu_j v_j} = \mu_j (v_j \otimes v_j)$.
\end{lemma}

\begin{proof}
Let $\sigma = [a_0, \ldots, a_k]$ and assume that $a_0$ is at the origin. First let's show that $\zeta \ \llcorner \ d\eta_i$ lies in the tangent space of $\sigma$. 
%One can complete an orthonormal basis $\{e_i\}_{i=1}^n$ such that $V = \text{span}(e_1, \ldots, e_k)$ is the same as span of $a_1, \ldots, a_k$. Suppose $\beta \in V^{\perp}$, then 
Let $V$ be the span of $a_1, \ldots, a_k$. For any $\beta \in V^{\perp}$, 
\begin{eqnarray*}
d\beta(\zeta \ \llcorner \ d\eta_i) &=& (d\eta_i \wedge d\beta) (\zeta) \\
&=& \text{det} 
\left[ \begin{matrix} (a_1 - a_0) \cdot (a_1 -a_0) & \cdots & (a_1-a_0) \cdot (a_k - a_0)\\
\vdots & & \vdots \\
\beta \cdot (a_1 -a_0) & \cdots & \beta \cdot (a_k - a_0)\\
 \end{matrix} \right] = 0.
\end{eqnarray*}
So $\zeta \ \llcorner \ d\eta_i \in V$, as desired. 

Next we want to show that $\zeta \ \llcorner \ d\eta_i$ is perpendicular to the $i$th face $s_i = [a_0,\ldots, \hat{a_i}, \ldots, a_k]$. If $i > 0$, for any $a_j - a_0$ with $j \neq i$, one has 
\begin{equation*}
d\eta_i \wedge d(a_j - a_0) = 0 \Rightarrow d(a_j - a_0)(\zeta \ \llcorner \ d\eta_i) =0.
\end{equation*}
For $i = 0$, the proof is similar. 

Then one can check that $(-1)^{i}\zeta \ \llcorner \ d\eta_i$ points outward. 

Lastly we look for the norm of $\zeta \ \llcorner \ d\eta_i$: 
\begin{eqnarray*}
&& \vert \zeta \ \llcorner \ d\eta_i \vert^2 \\
&=& d(\zeta \ \llcorner \ d\eta_i)(\zeta \ \llcorner \ d\eta_i) = d\eta_i \wedge d(\zeta \ \llcorner \ d\eta_i)(\zeta)\\
&=&(-1)^{k-i}  \text{det} 
\left[ \begin{matrix} (a_1 - a_0) \cdot (a_1 -a_0) & \cdots & (a_1-a_0) \cdot (a_i - a_0) & \cdots & (a_1 - a_0) \cdot (a_k-a_0)\\
\vdots & & & \vdots \\
(\zeta \ \llcorner \ d\eta_i) \cdot (a_1-a_0) &\cdots & (\zeta \ \llcorner \ d\eta_i) \cdot (a_i-a_0) & \cdots & (\zeta \ \llcorner \ d\eta_i) \cdot (a_k-a_0) \\
\vdots & & & \vdots \\
(a_k - a_0) \cdot (a_1 -a_0) & \cdots & (a_k-a_0) \cdot (a_i - a_0) & \cdots & (a_k - a_0) \cdot (a_k-a_0)
 \end{matrix} \right]\\
&=& (-1)^{k-i}  (\zeta \ \llcorner \ d\eta_i) \cdot (a_i-a_0) (\text{area of $s_i$} )^2
\end{eqnarray*}
Therefore
\begin{equation*}
 (\zeta \ \llcorner \ d\eta_i) \cdot (a_i-a_0) = \frac{(-1)^{k-i}\vert \zeta \ \llcorner \ d\eta_i\vert^2 }{(\text{area of $s_i$} )^2}.
\end{equation*}
On the other hand, 
\begin{eqnarray*}
&&  (\zeta \ \llcorner \ d\eta_i) \cdot (a_i-a_0)\\
&=& d(a_i - a_0)(\zeta \ \llcorner \ d\eta_i) = d\eta_i \wedge d(a_i - a_0)(\zeta)\\
&=&(-1)^{k-i} \text{det} 
\left[ \begin{matrix} (a_1 - a_0)\cdot (a_1 -a_0) & \cdots & (a_1-a_0)\cdot (a_k - a_0)\\
\vdots & & \vdots \\
(a_k - a_0) \cdot (a_1 -a_0) & \cdots & (a_k -a_0) \cdot (a_k - a_0)\\
 \end{matrix} \right] \\
 &=& (\text{volume of $\sigma$})^2, 
\end{eqnarray*}
where $\sigma = [a_0, \ldots, a_k]$. Thus
\begin{equation*}
\vert \zeta \ \llcorner \ d\eta_i \vert = (\text{area of $s_i$} )(\text{volume of $\sigma$}).
\end{equation*}
It follows that the force on $s_i$ coming from the $j$th Cauchy stress tensor $D_{\mu_j v_j}$ is
\begin{equation*}
\left [(-1)^i \frac{\zeta \ \llcorner \ d\eta_i }{\vert \zeta \ \llcorner \ d\eta_i \vert} \bullet v_j \right]  \mu_j v_j (\text{area of $s_i$}) 
= \frac{1}{\text{volume of $\sigma$}}(-1)^i \mu_j v_j [(\zeta \ \llcorner \ d\eta_i ) \bullet v_j].
\end{equation*}
\end{proof}

\subsubsection{Radial} 
The crux for proving the existence is to consider the special case when every simplex in the stressed polyhedral $(k-1)$-chain $Q$ contains the origin $O$. The idea is proof by induction and constructing cones. It is sufficient to understand the first lemma, since the last two are generalizations to higher dimensions.

\begin{lemma}\label{ycc:lemk=2}
Suppose $k=2$.  
%and $\mathbf{F}$ is given by (\ref{ycc:asystemofforcesk-1}) satisfying (\ref{ycc:equilibriumk-1}). 
Assume that 
\begin{equation}\label{ycc:F1O}
\mathbf{F} = \displaystyle\sum_{i=1}^l \mathbf{F}_i d\mathcal{H}^1_{\vert [O, a_i]},
\end{equation}
%the 1-simplices and the forces $\mathbf{F}_i$ are all contained in the subspace $\mathbb{R}^{n-1}$. Furthermore,
where each 1-simplex $[O, a_i]$ emanates from the origin $O$ to some $a_i \in \mathbb{R}^{n}$. Given that $\mathbf{F}$ satisfies the equilibrium condition in (\ref{ycc:equilibriumk-1}), then $\mathbf{F}$ is a linear combination of finitely many 2-beams.
\end{lemma}
%Note here we don't need to make a cone in the vertical axis, because overlapping is fine here. 
\begin{proof}
For each $[O, a_i]$, denote its barycenter by $\hat{a_i}$ which is equal to $a_i/2$. Then using these centers together with the forces, one can decompose them into a linear combination of 1-beams. See Figure \ref{fig:cone}. More precisely, let $\tilde{\mathbf{F}}$ be defined as:
\begin{equation*}
\tilde{\mathbf{F}} = \displaystyle\sum_{i=1}^l \mathbf{F}_i \delta_{\hat{a_i}},
\end{equation*}
which is also in equilibrium by (\ref{ycc:equilibriumk-1}). Therefore there exist $\tilde{l} \geq l$ points of application $\hat{a_1}, \ldots, \widehat{a_{\tilde{l}}}$ in $\mathbb{R}^{n}$ and an $\tilde{l} \times \tilde{l}$ symmetric matrix $\{\lambda_{ij}\}_{i, j = 1}^{\tilde{l}}$ of null diagonal such that 
\begin{equation*}
\tilde{\mathbf{F}} = \displaystyle\sum_{1 \leq i < j \leq \tilde{l}} \lambda_{ij} Beam([\hat{a_i}, \hat{a_j}]).
\end{equation*}
Note that the centers $\{\hat{a_1}, \ldots, \hat{a_l}\}$ are included in the set $\{\hat{a_1}, \ldots, \widehat{a_{\tilde{l}}}\}$. Moreover, one may change the point of reference in the proof of Gangbo to ensure that the origin $O$ is not any point of application. So we can connect the origin $O$ with each $\hat{a_i}$ and extend the ray to a new point $a_i$, so that $\hat{a_i}$ is the center of the 1-simplex $[O, a_i]$. It follows that the $[\hat{a_i}, \hat{a_j}]$ is not only parallel to $[a_i, a_j]$ but also lies in the tangent space of $[O, a_i, a_j]$. So there is a Cauchy stress tensor $A_{ij} = D_{\mu_{ij} v_{ij}}$ such that 
\begin{eqnarray*}
&&A_{ij} \otimes Beam([O, a_i, a_j])  \\
&=& \lambda_{ij}\frac{a_j - a_i}{\vert a_j - a_i\vert} d\mathcal{H}^1_{\vert [O, a_j]} - \lambda_{ij}\frac{a_j - a_i}{\vert a_j - a_i\vert} d\mathcal{H}^1_{\vert [O, a_i]} + 0 \ d\mathcal{H}^1_{\vert [a_i, a_j]}, 
\end{eqnarray*}
where 
\begin{equation*}
v_{ij} = \frac{a_j - a_i}{\vert a_j - a_i\vert}, \ \ \mu_{ij} = -\lambda_{ij} \cdot \frac{\vert a_j - a_i\vert}{\text{area of $[O, a_i, a_j]$}}.
\end{equation*}
Here we assume that $O < a_i < a_j$, otherwise we need to adjust the sign of $A_{ij}$.

Finally let's verify that 
\begin{equation*}
\mathbf{F} = \displaystyle\sum_{1 \leq i < j \leq \tilde{l}} A_{ij} \otimes Beam([O, a_i, a_j]). 
\end{equation*}
For each $[O, a_i]$, the total force due to the 2-beams is equal to the coefficent of $\hat{a_i}$ in $\tilde{\mathbf{F}}$, which is $\mathbf{F}_i$ when $i \leq l$ and zero when $i>l$. For each $[a_i, a_j]$, there is no force on it and its coefficient in $\tilde{\mathbf{F}}$ is also zero because there is no term for $\delta_{\frac{a_i + a_j}{2}}$. Thus we obtain (\ref{ycc:F1O}) as desired. 
\end{proof}

\begin{figure}
    \centering
    \includegraphics[width=0.5\linewidth]{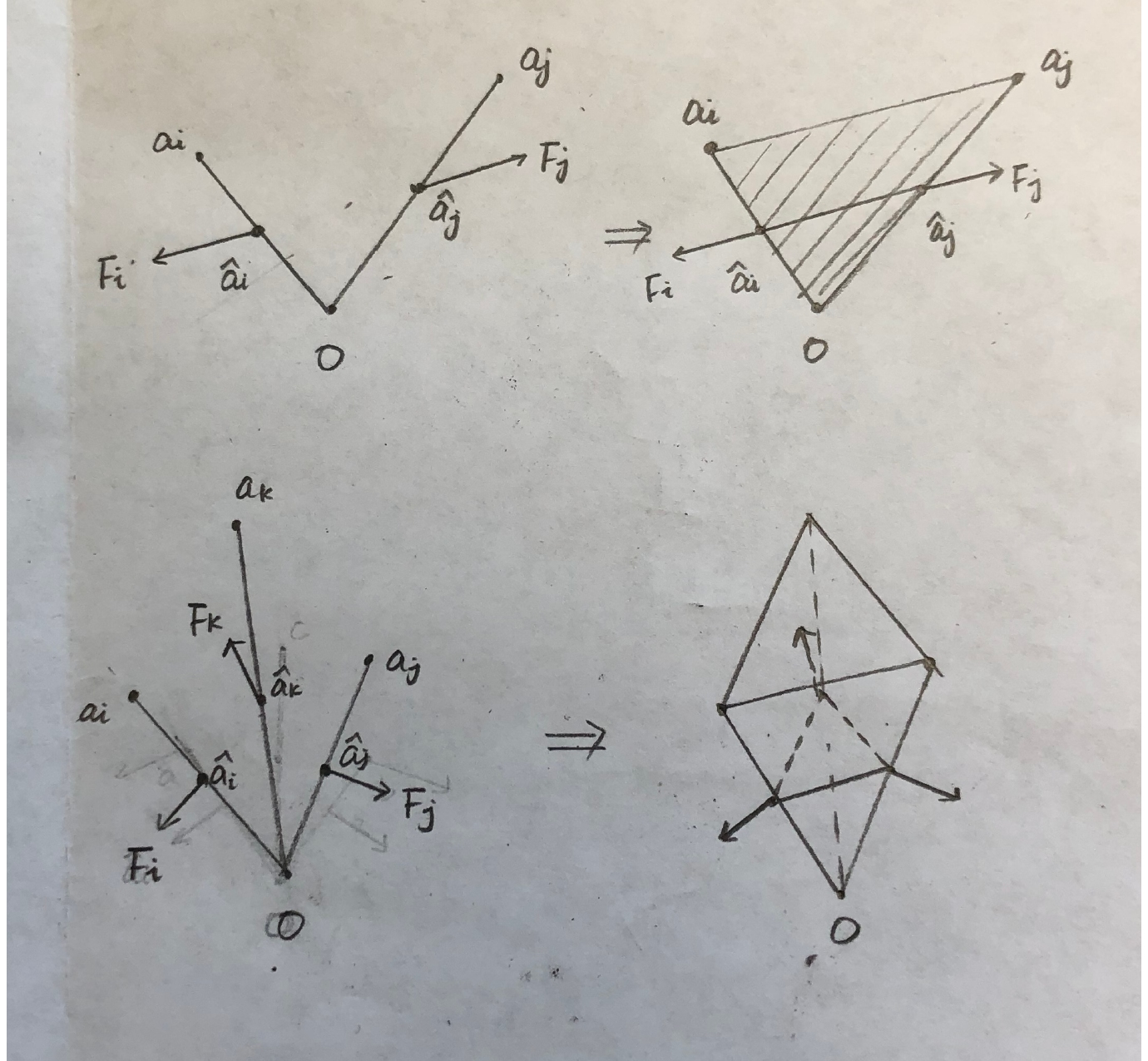}
    \caption{Make cones for radial polyhedral chains.}
    \label{fig:cone}
\end{figure}

\begin{lemma}\label{ycc:lemk=3}
Suppose $k=3$.  Assume that for any system of $1$-simplices and forces in $\mathbb{R}^n$ with $n \geq 3$, it can be decomposed into a finite combination of $2$-beams. 
%and $\mathbf{F}$ is given by (\ref{ycc:asystemofforcesk-1}) satisfying (\ref{ycc:equilibriumk-1}). 
Given that 
\begin{equation}\label{ycc:F2O}
\mathbf{F} = \displaystyle\sum_{i=1}^l \mathbf{F}_i d\mathcal{H}^2_{\vert [O, a_i, b_i]},
\end{equation}
%the 1-simplices and the forces $\mathbf{F}_i$ are all contained in the subspace $\mathbb{R}^{n-1}$. Furthermore,
where each 2-simplex $[O, a_i, b_i]$ emanates from the origin $O$ to some interval $[a_i, b_i]$. Given that $\mathbf{F}$ satisfies the equilibrium condition in (\ref{ycc:equilibriumk-1}), then $\mathbf{F}$ is a linear combination of finitely many 3-beams.
\end{lemma}

\begin{proof}
Let $\hat{a_i}$ be $2a_i/3$ and $\hat{b_i} = 2b_i/3$ for each $i$. 
Define $\tilde{\mathbf{F}}$ as 
\begin{equation*}
\tilde{\mathbf{F}} = \displaystyle\sum_{i=1}^l \mathbf{F}_i d\mathcal{H}^1_{\vert[\hat{a_i}, \hat{b_i}]},
\end{equation*}
then it is a system of forces in equilibrium because the center of $[\hat{a_i}, \hat{b_i}]$ is the same as that of $[O, a_i, b_i]$. 

By the hypothesis, there exist $\tilde{l} \geq  l$ intervals of application $[\hat{a_1}, \hat{b_1}], \ldots, [\widehat{a_{\tilde{l}}}, \widehat{b_{\tilde{l}}}]$ in $\mathbb{R}^n$ and real symmetric matrices $A_1, \ldots, A_{\tilde{l}}$ such that
\begin{equation*}
\tilde{\mathbf{F}} = \displaystyle\sum_{i=1}^{\tilde{l}} A_i \otimes Beam([\hat{a_i}, \hat{b_i}, \hat{c_i}]), 
\end{equation*}
where $\hat{c_i}$ can be either $\hat{a_j}$ or $\hat{b_j}$ for some $j \neq i$ as long as $[\hat{a_i}, \hat{b_i}], [\hat{b_i}, \hat{c_i}], [\hat{c_i}, \hat{a_i}]$ are contained in the set $\{[\hat{a_i}, \hat{b_i}]\}_{i=1}^{\tilde{l}}$. Note that the original intervals $\{[\hat{a_i}, \hat{b_i}]\}_{i=1}^l$ are also contained in the this set. 
For convenience one may use the notation:
\begin{equation*}
A_i \otimes Beam([\hat{a_i}, \hat{b_i}, \hat{c_i}]) = \mathbf{F}_{i, [\hat{a_i}, \hat{b_i}]}d\mathcal{H}^1_{\vert [\hat{a_i}, \hat{b_i}]} + \mathbf{F}_{i, [\hat{b_i}, \hat{c_i}]} d\mathcal{H}^1_{\vert [\hat{b_i}, \hat{c_i}]} + \mathbf{F}_{i, [\hat{c_i}, \hat{a_i}]} d\mathcal{H}^1_{\vert [\hat{c_i}, \hat{a_i}]},
\end{equation*}
where each $\mathbf{F}_i$ can be calculated using the formula in (\ref{ycc:kbeam}). For example, suppose $A_i = D_{\mu_{i, 1}v_{i, 1}} + D_{\mu_{i, 2}v_{i, 2}}$, then 
\begin{eqnarray*}
\mathbf{F}_{i, [\hat{a_i}, \hat{b_i}]} =\frac{1}{\text{area of $[\hat{a_i}, \hat{b_i}, \hat{c_i}]$}}\displaystyle\sum_{j=1}^2 \mu_{i, j}v_{i, j} [(\zeta_i \llcorner d(\hat{b_i} - \hat{a_i})) \bullet v_{i, j}],
\end{eqnarray*}
where $\zeta_i = (\hat{b_i} - \hat{a_i}) \wedge (\hat{c_i} - \hat{a_i})$. 

Moreover we may ensure that the origin $O$ does not lie in any of the 2-simplices by choosing new points of reference in the proof for the case $k=2$, together with $n \geq 3$.

So connect the origin $O$ with each of the $\hat{a_i}, \hat{b_i}, \hat{c_i}$ and extend the ray to a new point $a_i, b_i, c_i$, respectively, such that $\hat{a_i}, \hat{b_i}, \hat{c_i}$ are equal to $2/3$ of $a_i, b_i, c_i$ correspondingly. It follows that not only $[\hat{a_i}, \hat{b_i}, \hat{c_i}]$ is parallel to $[a_i, b_i, c_i]$ but also lies in the tangent space of $[O, a_i, b_i, c_i]$.

Therefore one can find a Cauchy stress  tensor $B_i$ such that 
\begin{eqnarray*}
&&B_i \otimes Beam([O, a_i, b_i, c_i]) \\
&=& \mathbf{F}_{i, [\hat{a_i}, \hat{b_i}]}d\mathcal{H}^2_{\vert [O, a_i, b_i]} + \mathbf{F}_{i, [\hat{b_i}, \hat{c_i}]} d\mathcal{H}^2_{\vert [O, b_i, c_i]} + \mathbf{F}_{i, [\hat{c_i}, \hat{a_i}]} d\mathcal{H}^2_{\vert [O, c_i, a_i]}
+ 0 \ d\mathcal{H}^2_{\vert[a_i, b_i, c_i]}.
\end{eqnarray*}
Here $B_i = D_{\lambda_{i, 1}v_{i, 1}} + D_{\lambda_{i,2}v_{i, 2}}$, where the $v_{i, j}$ are eigenvectors in $A_i$, and 
\begin{equation*}
\lambda_{i, j} = -\frac{2}{3} \mu_{i, j} \cdot \frac{\text{area of $[a_i, b_i, c_i]$}}{\text{volume of $[O, a_i, b_i, c_i]$}}.
\end{equation*}
Finally one can verify that
\begin{equation*}
\mathbf{F} = \displaystyle\sum_{i=1}^{\tilde{}l} B_i \otimes Beam([O, a_i, b_i, c_i]).
\end{equation*}
\end{proof}

Now let's generalize the proof of the previous lemma to any $k \geq 3$. 
\begin{lemma}\label{ycc:lemmaO}
Suppose $k \geq 3$. Assume that for any system (\ref{ycc:asystemofforcesk-1}) of $(k-2)$-simplices and forces in $\mathbb{R}^n$ with $n \geq k$ and satisfying (\ref{ycc:equilibriumk-1}), it can be decomposed into a finite combination of $(k-1)$-beams. Then given that
\begin{equation}\label{ycc:FkO}
\mathbf{F} = \displaystyle\sum_{i=1}^l \mathbf{F}_i d\mathcal{H}^{k-1}_{\vert [O, a_0^i, \ldots, a_{k-2}^i]},
\end{equation}
%the 1-simplices and the forces $\mathbf{F}_i$ are all contained in the subspace $\mathbb{R}^{n-1}$. Furthermore,
where each $(k-1)$-simplex $[O, a_0^i, \ldots, a_{k-2}^i]$ emanates from the origin $O$ to some $(k-2)$-simplex $[a_0^i, \ldots, a_{k-2}^i]$. Given that $\mathbf{F}$ satisfies the equilibrium condition in (\ref{ycc:equilibriumk-1}), then $\mathbf{F}$ is a linear combination of finitely many $k$-beams.
\end{lemma}

\begin{proof}
For each $1 \leq i \leq l$ and $0 \leq j \leq k-2$, let
\begin{equation*}
\hat{a_j^i} = \frac{k-1}{k}a_j^i,
\end{equation*}
so that the center of $[\hat{a_0^i}, \ldots, \widehat{a_{k-2}^i}]$ equals that of $[O, a_0^i, \ldots, a_{k-2}^i]$. Define $\tilde{\mathbf{F}}$ as
\begin{equation*}
\tilde{\mathbf{F}} = \displaystyle\sum_{i=1}^l \mathbf{F}_i d\mathcal{H}^{k-2}_{\vert [\hat{a_0^i}, \ldots, \widehat{a_{k-2}^i}]},
\end{equation*}
and it is an equilibrium system of forces. 

By the hypothesis, there exist $\tilde{l}$, with $\tilde{l} \geq l$, $(k-2)$-simplices of application $\{[\hat{a_0^i}, \ldots, \widehat{a_{k-2}^i}]\}_{i=1}^{\tilde{l}}$ and real symmetric matrices $\{A_i\}_{i=1}^{\tilde{l}}$ such that
\begin{equation*}
\tilde{\mathbf{F}} = \displaystyle\sum_{i=1}^{\tilde{l}} A_i \otimes Beam([\hat{a_0^i}, \ldots, \widehat{a_{k-2}^i},\widehat{a_{k-1}^i}]),
\end{equation*}
where $\widehat{a_{k-1}^i}$ can be any of the vertices of $[\hat{a_0^m}, \ldots, \widehat{a_{k-2}^m}]$ for some $m \neq i$ so that the following $(k-2)$-simplices 
\begin{equation*}
[\hat{a_0^i}, \ldots, \widehat{a_{k-2}^i}], [\hat{a_1^i}, \ldots, \widehat{a_{k-1}^i}], \ldots, [\hat{a_0^i}, \ldots, \widehat{a_{k-3}^i}, \widehat{a_{k-1}^i}]
\end{equation*}
are contained in the set $\{[\hat{a_0^i}, \ldots, \widehat{a_{k-2}^i}]\}_{i=1}^{\tilde{l}}$.
Note that the original $(k-2)$-simplices $\{[\hat{a_0^i}, \ldots, \widehat{a_{k-2}^i}]\}_{i=1}^{l}$ are contained in this set as well. For convenience we write:
\begin{eqnarray*}
&&A_i \otimes Beam([\hat{a_0^i}, \ldots, \widehat{a_{k-2}^i},\widehat{a_{k-1}^i}]) \\
&=& \displaystyle \sum_{j=0}^{k-1} \mathbf{F}_{i, j} d\mathcal{H}^{k-2}_{\vert[\hat{a_0^i}, \ldots, \hat{a_{j-1}^i}, \hat{a_{j+1}^i}, \ldots, \hat{a_{k-1}^i}]}, 
\end{eqnarray*}
where each $\mathbf{F}_{i, j}$ can be calculated using the formula in (\ref{ycc:kbeam}). For example, suppose that $A_i = \sum_{m=0}^{k-2}D_{\mu_{i, m}v_{i, m}}$, then 
\begin{equation*}
\mathbf{F}_{i, j} = \frac{1}{\text{volume of $[\hat{a_0^i}, \ldots,\hat{a_{k-1}^i}]$}} \displaystyle\sum_{m=0}^{k-2}\mu_{i, m}v_{i, m}[(\zeta_i \ \llcorner \ d\eta_{i, j})\bullet v_{i, m}],
\end{equation*}
where 
\begin{eqnarray*}
\zeta_i &=& (\hat{a_1^i}-\hat{a_0^i}) \wedge \cdots \wedge (\hat{a_{k-1}^i}-\hat{a_0^i}), \\
d\eta_{i, j} &=& d(\hat{a_1^i}-\hat{a_0^i}) \wedge \cdots d(\hat{a_{j-1}^i}-\hat{a_0^i}) \wedge d(\hat{a_{j+1}^i}-\hat{a_0^i}) \cdots \wedge d(\hat{a_{k-1}^i}-\hat{a_0^i}). 
\end{eqnarray*}
Moreover we may ensure that the origin $O$ does not lie any of the $(k-1)$-simplices by choosing a new point of reference in the proof for the $k-1$ case under the assumption that $n \geq k$.
 
Next connect the origin $O$ with each $\hat{a_j^i}$ and extend the ray to a new point $a_j^i$ such that 
\begin{equation*}
a_j^i = \frac{k}{k-1}\hat{a_j^i}.
\end{equation*}
It follows that not only $[\hat{a_0^i}, \ldots,\hat{a_{k-1}^i}]$ is parallel to $[a_0^i, \ldots, a_{k-1}^i]$ but also lies in the tangent space of $[O, a_0^i, \ldots, a_{k-1}^i]$.
Therefore one can find a Cauchy stress tensor $B_i$ such that 
\begin{eqnarray*}
&&B_i \otimes Beam([O, a_0^i, \ldots, a_{k-1}^i]) \\
&=& \displaystyle \sum_{j=0}^{k-1} \mathbf{F}_{i, j} d\mathcal{H}^{k-1}_{\vert[O, a_0^i, \ldots, a_{j-1}^i, a_{j+1}^i, \ldots, a_{k-1}^i]}.
\end{eqnarray*}
Here $B_i = \sum_{m=0}^{k-2}D_{\lambda_{i, m}v_{i, m}}$ with the $v_{i, m}$ being the eigenvectors of $A_i$ and 
\begin{equation*}
\lambda_{i, m} = -\left(\frac{k-1}{k}\right)^{k-2}\mu_{i,m} \cdot \frac{\text{area of $[a_0^i, \ldots, a_{k-1}^i]$}}{\text{volume of $[O, a_0^i, \ldots, a_{k-1}^i]$}}.
\end{equation*}
Finally let's verify that 
\begin{equation*}
\mathbf{F} = \displaystyle\sum_{i=1}^{\tilde{l}} B_i \otimes Beam([O, a_0^i, \ldots, a_{k-1}^i]).
\end{equation*}
For each $[O, a_0^i, \ldots, a_{j-1}^i, a_{j+1}^i, \ldots, a_{k-1}^i]$, the total force on it due to the $k$-beams is equal to the coefficient of $[\hat{a_0^i}, \widehat{\ldots, a_{j-1}^i},\widehat{a_{j+1}^i}, \ldots, \hat{a_{k-1}^i}]$ in $\tilde{\mathbf{F}}$ which is equal to $\mathbf{F}_i$ when $i \leq l$ and zero when $i > l$. For each $[a_0^i, \ldots, a_{k-1}^i]$ there is no force on it and its coefficient in $\tilde{\mathbf{F}}$ is also zero because there is no $d\mathcal{H}^{k-1}$ term for it. Thus we obtain (\ref{ycc:FkO}) as desired. 
\end{proof}

\subsubsection{Nonradial}
Given the above lemmas, one can prove the general result by induction. The following three propositions use some of the ideas of Gangbo.  

\begin{proposition}\label{ycc:prop1}
Suppose $n \geq k+2$. Assume that $\mathbf{F}$ is given by (\ref{ycc:asystemofforcesk-1}) and that 
\begin{itemize}
\item[(i)] $[a_0^i, \ldots, a_{k-1}^i] \in \mathbb{R}^{n-1}$ for each $i = 1, \ldots, l$;
\item[(ii)] $\mathbf{F}_i$ is perperdicular to the hyperplane $\mathbb{R}^{n-1}$.
\end{itemize}
Then one may decompose $\mathbf{F}$ into a sum of two equilibrium systems of forces 
\begin{equation*}
\mathbf{F} = \mathbf{B} + \mathbf{F}^h,
\end{equation*}
where $\mathbf{B}$ is a finite combination of $k$-beams, and $\mathbf{F}^h$ is an equilibrium system of $(k-1)$-simplices and forces in $\mathbb{R}^{n-1}$.
\end{proposition}

\begin{proof}
Fix $\bar{O} = e_n$. 
For each $s_i =[a_0^i, \ldots, a_{k-1}^i]$, denote its center by $\hat{s_i}$ which is contained in $\mathbb{R}^{n-1}$ by hypothesis. 
If $\mathbf{F}_i= f_ie_n$, then we can write $\mathbf{F}_i$ as
\begin{equation*}
\mathbf{F}_i = \mathbf{F}_i^{\bar{O}} + \mathbf{F}_i^h, 
\end{equation*}
where 
\begin{equation*}
\mathbf{F}_i^h = f_i \hat{s_i}, \ \ \mathbf{F}_i^{\bar{O}} = \mathbf{F}_i - \mathbf{F}_i^h.
\end{equation*}

Since $\mathbf{F}_i$ is parallel to $e_n$, the $n$th column of $\hat{s_i} \wedge \mathbf{F}_i$ is equal to $f_i\hat{s_i}$, therefore
\begin{equation*}
\displaystyle\sum_{i=1}^l \mathbf{F}_i^h = 0.
\end{equation*}
Moreover since $\hat{s_i} \wedge \hat{s_i} = 0$,
\begin{equation*}
\displaystyle\sum_{i=1}^l \mathbf{F}_i^h \wedge \hat{s_i} = 0.
\end{equation*}
So 
\begin{equation*}
\mathbf{F}^h = \displaystyle\sum_{i=1}^l \mathbf{F}_i^h d\mathcal{H}^{k-1}_{\vert s_i}
\end{equation*}
is a system in equilibrium whose simplices and forces are all contained in $\mathbf{R}^{n-1}$. Being the difference between two equilibrium systems, 
\begin{equation*}
\mathbf{F}^{\bar{O}} = \displaystyle\sum_{i=1}^l \mathbf{F}_i^{\bar{O}} d\mathcal{H}^{k-1}_{\vert s_i}
\end{equation*}
is also in equilibrium. 

Next we want to show that $\mathbf{F}_i^{\bar{O}}$ comes from finitely many $k$-beams. First notice that $\mathbf{F}_i^{\bar{O}} \in$ Span($a_0^i - \bar{O}, \ldots, a_{k-1}^i - \bar{O}$) because
\begin{equation*}
\mathbf{F}_i^{\bar{O}} = f_i(e_n - \hat{s_i}) = -\frac{f_i}{k}((a_0^i  - e_n) + \cdots + (a_{k-1}^i - e_n)).
\end{equation*}

%%%%%%%%%%%%%%%%%%%%%%%%%%%%%%%%%%%%%%%%%%
We are going to translate $s_i = [a_0^i, \ldots, a_{k-1}^i]$ in the direction of $\mathbf{F}_i^{\bar{O}}$ to arrive at a new $(k-1)$-simplex $\tilde{s_i} = [\tilde{a_0^i}, \ldots, \tilde{a_{k-1}^i}]$ so that 
\[\bar{O} = \sum_{j = 0}^{k-1} t_j \tilde{a_j^i}.\]
Indeed, one can compute that 
\[\bar{O} = \frac{1}{k}((a_0^i + \frac{\mathbf{F}_i^{\bar{O}}}{f_i}) + \cdots + (a_{k-1}^i + \frac{\mathbf{F}_i^{\bar{O}}}{f_i})),\]
so that 
\[\tilde{a_j^i} = a_j^i + \frac{\mathbf{F}_i^{\bar{O}}}{f_i}, \ \ t_j = \frac{1}{k}.\]
Construct the parallelepiped $P_i$ formed by using the base $s_i$ and the side $\frac{\mathbf{F}_i^{\bar{O}}}{f_i}$: that is, the $P_i$ is spanned by $a_1^i - a_0^i, \ldots, a_{k-1}^i - a_0^i, \frac{\mathbf{F}_i^{\bar{O}}}{f_i}$, and is based at $a_0^i$. Furthermore, let $A_i$ be the symmetric matrix $D_{\mu_iv_i}$, where 
\[v_i = \frac{\mathbf{F}_i^{\bar{O}}}{\vert \mathbf{F}_i^{\bar{O}} \vert}, \ \ \mu_i = \frac{\vert \mathbf{F}_i^{\bar{O}} \vert^2}{f_i \text{area of $s_i$}}.\]
Since $A_i$ contributes to zero forces on all $(k-1)$-faces of $P_i$ except $s_i$ and $\tilde{s_i}$, 
\[\mathbf{F}_i^{\bar{O}} d\mathcal{H}^{k-1}_{\vert s_i} - \mathbf{F}_i^{\bar{O}} d\mathcal{H}^{k-1}_{\vert \tilde{s_i}} = A_i \otimes Beam(P_i),\]
where $P_i$ can be decomposed into finitely many $k$-simplices. 
Furthermore, since $\bar{O}$ is the barycenter of $\tilde{s_i}$, one may rewrite $\tilde{s_i}$ as a sum of subsimplices as follows:
\[\tilde{s_i} = \sum_{j=0}^{k-1} [\tilde{a_0^i}, \ldots, \bar{O}, \ldots, \tilde{a_{k-1}^i}], \]
where $\bar{O}$ is in the $j$th spot.

%% A mistake here because A_i might not be symmetric.
%Construct the $k$-simplex $[\bar{O}, a_0^i, \ldots, a_{k-1}^i]$ for each $s_i$. We may ``push" $s_i$ to another face $[\bar{O}, a_0^i, \ldots, a_{k-2}^i]$ using the $\mathbf{F}_i^{\bar{O}}$ on $s_i$ and zero forces on 
%\[[\bar{O}, a_0^i, \ldots, a_{j-1}^i, a_{j+1}^i, \ldots, a_{k-1}^i]\]
%for $0 \leq j \leq k-2$. In other words, there exists a symmetric matrix $A_i$ and  a force $\mathbf{F}_i^{\bar{O}'}$ such that 
%\begin{eqnarray*}
%&&A_i \otimes Beam([\bar{O}, a_0^i, \ldots, a_{k-1}^i]) \\
%&=& \mathbf{F}_i^{\bar{O}} d\mathcal{H}^{k-1}_{\vert s_i} + \displaystyle\sum_{j=0}^{k-2} 0 \ d\mathcal{H}^{k-1}_{\vert [\bar{O}, a_0^i, \ldots, a_{j-1}^i, a_{j+1}^i, \ldots, a_{k-1}^i]} + \mathbf{F}_i^{\bar{O}'} d\mathcal{H}^{k-1}_{\vert [\bar{O}, a_0^i, \ldots, a_{k-2}^i]}.
%\end{eqnarray*}
Hence $\mathbf{F}^{\bar{O}}$ can be decomposed into another two systems: one consists of $k$-beams and the other all connecting with $\bar{O}$. That is
\begin{equation}\label{ycc:propperppush}
\mathbf{F}^{\bar{O}} = \displaystyle\sum_{i=1}^l A_i \otimes Beam(P_i) + \mathbf{F}^{\bar{O}'}
\end{equation}
where
\begin{equation*}
\mathbf{F}^{\bar{O}'} = \displaystyle\sum_{i=1}^l \sum_{j=0}^{k-1} \mathbf{F}_i^{\bar{O}} d\mathcal{H}^{k-1}_{\vert [\tilde{a_0^i}, \ldots, \bar{O}, \ldots, \tilde{a_{k-1}^i}]}.
\end{equation*}

According to the lemma \ref{ycc:lemmaO}, there are symmetric matrices $B_i$ such that 
\begin{equation}\label{ycc:propperpO}
\mathbf{F}^{\bar{O}'} = \displaystyle\sum_{i=1}^{\tilde{l}} B_i \otimes Beam([\bar{O}, a_0^i, \ldots, a_{k-1}^i]).
\end{equation} 
Let $\mathbf{B}$ be the sum (\ref{ycc:propperppush}) + (\ref{ycc:propperpO}). Then $\mathbf{B}$ is a finite combination of $k$-beams and 
\begin{equation*}
\mathbf{F}^{\bar{O}} = \mathbf{B},
\end{equation*}
as desired. 
\end{proof}

\begin{proposition}\label{ycc:prop2}
Suppose $n \geq k+2$. Assume that $\mathbf{F}$ is given by (\ref{ycc:asystemofforcesk-1}) and that 
\begin{itemize}
\item[(i)] $[a_0^i, \ldots, a_{k-1}^i] \in \mathbb{R}^{n-1}$ for each $i = 1, \ldots, l$;
\item[(ii)] $\mathbf{F}_i \in \mathbb{R}^{n}$.
\end{itemize}
Then one may decompose $\mathbf{F}$ into a sum of two equilibrium systems of forces 
\begin{equation*}
\mathbf{F} = \mathbf{B} + \mathbf{F}^h,
\end{equation*}
where $\mathbf{B}$ is a finite combination of $k$-beams, and $\mathbf{F}^h$ is a equilibrium system of $(k-1)$-simplices and forces in $\mathbb{R}^{n-1}$.
\end{proposition}

\begin{proof}
Decompose each $\mathbf{F}_i$ into $\mathbf{F}_i^h$, that is parallel to $\mathbb{R}^{n-1}$, and $\mathbf{F}_i^v = f_ie_n$, for some real number $f_i$. So we obtain two systems as follows:
\begin{eqnarray*}
\mathbf{F}^h &=& \displaystyle\sum_{i=1}^l \mathbf{F}_i^h d\mathcal{H}^{k-1}_{\vert s_i};\\
\mathbf{F}^v &=& \displaystyle\sum_{i=1}^l f_ie_n d\mathcal{H}^{k-1}_{\vert s_i},
\end{eqnarray*}
where
\begin{equation*}
\mathbf{F} = \mathbf{F}^h + \mathbf{F}^v.
\end{equation*}
Since the $n$th column of $\mathbf{F}_i^h \wedge \hat{s}_i$ is equal to zero for each $i$, 
\begin{equation*}
\displaystyle\sum_{i=1}^l f_ie_n \wedge \hat{s}_i.
\end{equation*}
Furthermore $\sum_{i=1}^l f_i = 0$ so that $\mathbf{F}_i^v$ is a system in equilibrium, which implies that $\mathbf{F}_i^h$ is also a system in equilibrium.
Applying Proposition \ref{ycc:prop1} to $\mathbf{F}_i^v$ gives us the desired result.
\end{proof}

\begin{proposition}\label{ycc:prop3}
Suppose $n \geq k+2$. Assume that $\mathbf{F}$ is given by (\ref{ycc:asystemofforcesk-1}) and that 
\begin{itemize}
\item[(i)] $[a_0^i, \ldots, a_{k-1}^i] \in \mathbb{R}^{n}$ for each $i = 1, \ldots, l$;
\item[(ii)] $\mathbf{F}_i \in \mathbb{R}^{n}$.
\end{itemize}
Then one may decompose $\mathbf{F}$ into a sum of two equilibrium systems of forces 
\begin{equation*}
\mathbf{F} = \mathbf{B} + \mathbf{F}^h,
\end{equation*}
where $\mathbf{B}$ is a finite combination of $k$-beams, and $\mathbf{F}^h$ is a equilibrium system of $(k-1)$-simplices and forces in $\mathbb{R}^{n-1}$.
\end{proposition}

\begin{proof}
To bring each simplex down to $\mathbb{R}^{n-1}$, we use the following two steps. 

Step 1: Suppose $s_i$ is parallel to the hyperplane $\mathbb{R}^{n-1}$, let's decompose $\mathbf{F}_i$ into its vertical and horizontal components:
\begin{equation*}
\mathbf{F}_i = \mathbf{F}_i^h + \mathbf{F}_i^v.
\end{equation*}
First for the vertical component $\mathbf{F}_i^v$, one can translate $s_i$ vertically so that it lies in the hyperplane. More precisely, let $d$ be such that $a_0^i - de_n \in \mathbb{R}^{n-1}$. Denote $a_j^i - ed_n$ by $\tilde{a_j^i}$ for each $j$ between 0 and $k-1$, and $[\tilde{a_0^i}, \ldots, \tilde{a_{k-1}^i}]$ by $\tilde{s_i}$. Let $I = [0, d]$. Then $s_i \times I$ can be subdivided into $k$-simplices and so we may ``push" $s_i$ to $\tilde{s_i}$. More precisely, there exists a symmetric matrix $A_i$ such that
\[A_i = D_{\mu_i v_i}, \text{ where } \mu_i = \frac{\vert \mathbf{F}_i^v \vert}{\text{ area of $s_i$}}, v_i = \frac{\mathbf{F}_i^v}{\vert \mathbf{F}_i^v \vert},\]
and
\begin{eqnarray*}
&&\mathbf{F}_i^v d\mathcal{H}^{k-1}_{\vert s_i} - \mathbf{F}_i^v d\mathcal{H}^{k-1}_{\vert \tilde{s_i}}\\
&=& A_i \otimes Beam(s_i \times I) \\
&=& A_i \otimes Beam[\tilde{a_0^i}, a_0^i, \ldots, a_{k-1}^i] - A_i \otimes Beam[\tilde{a_0^i}, \tilde{a_1^i}, a_1^i, \ldots, a_{k-1}^i] \\
&&+ \cdots + (-1)^{k-1} A_i \otimes Beam[\tilde{a_0^i}, \tilde{a_1^i}, \ldots, \tilde{a_{k-1}^i},a_{k-1}^i].
\end{eqnarray*}

Next for the horizontal component $\mathbf{F}_i^h$, we break it into another two forces:
\begin{equation*}
\mathbf{F}_i^{h, 1} = \mathbf{F}_i^h + e_n, \ \ \mathbf{F}_i^{h, 2} = \mathbf{F}_i^h - e_n.
\end{equation*}

Then one may ``push" $s_i$ to another $\tilde{s_i}^j \in \mathbb{R}^{n-1}$ in the direction of $\mathbf{F}_i^{h, j}$ for $j=1, 2$. This is analogous to the previous case.  Let $P_i^j$ be the parallelepiped spanned by pushing $s_i$ to $\tilde{s_i}^j \in \mathbb{R}^{n-1}$ in the direction of $\mathbf{F}_i^{h, j}$ for $j = 1, 2$; moreover, let $A_i^{h, j} = D_{\mu_i^{h, j}v_i^{h, j}}$ be defined as
\[\mu_i^{h, j} = \frac{\vert \mathbf{F}_i^{h, j} \vert^2}{(e_n \cdot \mathbf{F}_i^{h, j}) \text{ area of $s_i$}} , \ \ v_i^{h, j} = \frac{\mathbf{F}_i^{h, j}}{\vert  \mathbf{F}_i^{h, j} \vert}.\]
It follows that for each $j$,
\[\mathbf{F}_i^{h, j} d\mathcal{H}^{k-1}_{\vert s_i} - \mathbf{F}_i^{h, j} d\mathcal{H}^{k-1}_{\vert \tilde{s_i}^j} = A_i^{h, j} \otimes Beam(P_i^j),\]
where $P_i^j$ is a finite union of $k$-simplices.

%%%%%%%%%%%%%%%%%%%%%%%%%%%%%%%%%%%%%%%%%%%%%%
Step 2: Now let us assume that $s_i$ is not parallel to the hyperplane $\mathbb{R}^{n-1}$. Without loss of generality, one may also assume that $\mathbf{F}_i$ is not in the tangent space of $s_i$ and the hyperplane $\mathbb{R}^{n-1}$, i.e., 
\[\mathbf{F}_i \not\in \text{span}(a_1^i - a_0^i, \ldots, a_{k-1}^i - a_0^i), \ \ \mathbf{F}_i \not \in \mathbb{R}^{n-1}.\]
Otherwise we may decompose $\mathbf{F}_i$ into a sum of two forces which don't. 
Then consider the parallelepiped with base $s_i$ and side parallel to $\mathbf{F}_i$.
It intersects $\mathbb{R}^{n-1}$ at some $(k-1)$-simplex $\tilde{s_i}$:
\[\tilde{s_i} = [\ldots, a_j^i - \frac{a_j^i \cdot e_n}{\mathbf{F}_i \cdot e_n}\mathbf{F}_i,\ldots].\]
Let $P_i$ be the polyhedron with top $s_i$, bottom $\tilde{s_i}$, and sides parallel to $\mathbf{F}_i$, then $P_i$ is a finite union of $k$-simplices. Moreover, we can find a symmetric matrix $B_i$ such that 
\[B_i \otimes Beam(P_i) = \mathbf{F}_i d\mathcal{H}^{k-1}_{\vert s_i} - \mathbf{F}_i d\mathcal{H}^{k-1}_{\vert \tilde{s_i}}.\]
Indeed $B_i = D_{\mu_iv_i}$:
\[\mu_i = \frac{\vert \mathbf{F}_i  \vert^2}{(e_n \cdot \mathbf{F}_i ) \text{area of $\tilde{s_i}$}} , \ \ v_i = \frac{\mathbf{F}_i }{\vert  \mathbf{F}_i \vert}.\]
Applying Proposition \ref{ycc:prop2} finishes the proof.
\end{proof}

\begin{theorem}\label{ycc:thm2}
Suppose $k \geq 1$ and $n \geq k+1$. Assume that $\mathbf{F}$ is given by (\ref{ycc:asystemofforcesk-1}) and $\mathbf{F}$ satisfies the equilibrium condition in (\ref{ycc:equilibriumk-1}), then $\mathbf{F}$ can be decomposed into finitely many $k$-beams.
\end{theorem}

\begin{proof}
We prove by induction on $k$. First when $k=1$, this is done in Gangbo's notes. Let $k \geq 2$  and assume that for any system of $(k-2)$-simplices and forces in $\mathbb{R}^n$ with $n \geq k$, it can be decomposed into a finite combination of $(k-1)$-beams. 

Now let's prove for the case when there are $(k-1)$-simplices with forces acting on them in $\mathbb{R}^n$, where $n \geq k+1$. To prove this case, we are going to induct on the dimension $n$ of $\mathbb{R}^n$. 

If $n=k+1$, there exists $\bar{O}$ that doesn't lie in the tangent space of any $(k-1)$-simplices, then we may push $s_i$ to another $\tilde{s_i}$ following the exact same proof in Proposition \ref{ycc:prop1}. Therefore $\mathbf{F}$ may be decomposed into 
a sequence of $k$-beams and another system $\mathbf{F}^{\bar{O}'}$ of $(k-1)-$simplices that all connect to $\bar{O}$. By the inductive hypothesis and Lemma \ref{ycc:lemmaO}, $\mathbf{F}^{\bar{O}'}$ is also a sequence of $k$-beams.

If $n \geq k+2$, Proposition \ref{ycc:prop1}, \ref{ycc:prop2}, and \ref{ycc:prop3} imply that $\mathbf{F}$ is a sequence $\mathbf{B}$ of $k$-beams and  an equilibrium $\mathbf{F}^h$ system of $(k-1)$-simplices and forces in $\mathbb{R}^{n-1}$. Applying the inductive hypothesis on $n-1$ yields the desired result. 
\end{proof}

\subsection{Matching Stresses}
Finally, we can solve the existence question in two steps. First, suppose $Q$ is a stressed polyhedral $(k-1)$-chain whose coefficients are generalized Cauchy stress tensors. Then $Q$ can be rewritten as
\[Q = S + F,\]
with $S$ being the stresses and $F$ being the external forces.

If we assume that the net force and torque in $F$ are both equal to zero,
then according to Theorem \ref{ycc:thm2}, there exists a structurally stressed polyhedral $k$-chain $P_1$ such that 
\[\partial P_1 = S_1 + F,\]
where the external forces of $P_1$ on its boundary can balance out the $F$ in $Q$. 
Therefore $ \partial P_1 - Q = S_1 - S$ has only internal stresses and no external forces, since the external forces in $F$ are balanced. Furthermore, since the coefficients are symmetric matrices in $Q$ and $P_1$, if there is no external parallel force, then there is no shear stress. This can be observed from the matrix decomposition in (\ref{ycc:tildeA}). 
 It follows that the coefficients of $S_1- S$ are only symmetric matrices. 
Therefore it remains to show that there exists a structurally stressed $k$-chain $P_2$ such that 
\[\partial P_2 = S_1 - S.\]
In fact, one can prove the following lemma.
\begin{lemma}\label{ycc:lem3}
Suppose $R$ is a structurally stressed $(k-1)$-chain in $\mathbb{R}^n$, then $ \partial R = 0$ if and only if there exists a structurally stressed $k$-chain $P$ such that $\partial P = R$.
\end{lemma}

\begin{proof}
The if part is trivial, so let us prove the only if part. The idea is to construct cones. First, one can express $R$ as
\[R = \sum_{i=1}^m A_i \otimes \sigma_i.\]
Since $k \leq n$, there exists a point $w \in \mathbb{R}^n$ such that $w$ does not lie in the tangent spaces of $\sigma_i$'s. So one may construct the cone $w \ast \sigma_i$ of $w$ over each $\sigma_i$. Define
\[P = \sum_{i=1}^m A_i \otimes (w \ast \sigma_i).\]
Since $R$ is structurally stressed, the eigenspaces of $A_i$ corresponding to nonzero eigenvalues lie inside the span($\sigma_i$). Therefore they also lie in the span($w \ast \sigma_i$). So $P$ is also structurally stressed. 

Next, since $\partial (w \ast \sigma) = \sigma - w \ast \partial \sigma$ and $\partial R = 0$, one can check that 
\[ \partial P = \sum_{i=1}^m  A_i \otimes \partial (w \ast \sigma_i) =  \sum_{i=1}^m A_i \otimes \sigma_i = R.\]
\end{proof}

\begin{theorem}\label{ycc:thmexistence}
Suppose $Q$ is a stressed polyhedral $(k-1)$-chain such that
\[Q = \sum_{i=1}^m A_i \otimes \sigma_i = S + F,\]
where the $A_i$ are generalized Cauchy stress tensor for each $\sigma_i$, $S$ consists of the stresses, and $F$ consists of the forces. Then there exists a structurally stressed polyhedral $k$-chain $P$ such that
\[\partial P = Q\]
if and only if $\partial Q = 0$ and the net force and the net torque of $F$ are both equal to zero.

\end{theorem}

\begin{proof}
The if direction is trivial, so let us focus on the only if direction. Since $A_i$ is a generalized Cauchy stress tensor for each $\sigma_i$, there exists an eigenvector $v_i$ of $A_i$ whose eigenvalue is not zero and $v_i$  does not lie in the tangent space of $\sigma_i$.  It contributes to the external force $\mathbf{F_i}$ on $\sigma_i$
\[\mathbf{F_i} = A_i v_i \cdot (\text{ area of $\sigma_i$}). \]
Let 
\[\mathbf{F} = \sum_{i=1}^m \mathbf{F_i} \mathcal{H}_{\vert_{\sigma_i}}^{k-1}.\]
By hypothesis, the system $\mathbf{F}$ satisfies the equilibrium conditions that 
\[\sum \mathbf{F_i} = 0, \ \ \sum \mathbf{F_i} \wedge \hat{\sigma_i} =0,\]
where $\hat{\sigma_i}$ is the barycenter of $\sigma_i$.
According to Theorem \ref{ycc:thm2},  there exists a structurally stressed $k$-chain $P_1$ such that $R = \partial P_1 - Q$
is a structurally stressed $(k-1)$-chain. Moreover, 
\[\partial R =  \partial^2 P_1 -  \partial Q = 0.\]
So by Lemma \ref{ycc:lem3}, there exists a structurally stressed $k$-chain $P_2$ such that $ \partial P_2 = R$. Combining $P_1$ with $P_2$, one yields
\[\partial (P_1 - P_2) =  R + Q - R = Q. \]
So $P = P_1 - P_2$ is the desired structurally stressed $k$-chain. 
\end{proof}

\newpage

\section{The Plateau Problem}
Given the existence property, the next natural question to study is optimization. 
We can put a norm $\vert \cdot \vert$ on symmetric matrices, so that $(Sym_n, \vert \cdot \vert)$ is  a {\it normed abelian group}. Then let's define the {\it mass} norm on a structural stressed polyhedral chain $P = \sum_{i=1}^m A_i \otimes \sigma_i$ as follows. 
\[\mathbb{M}(P) = \sum_{i=1}^m \vert A_i \vert \mathcal{H}^k(\sigma_i),\]
where $\mathcal{H}^k(\sigma_i)$ is the $k$-dimensional volume of $\sigma_i$. 

\noindent{\bf The Cost of Constructing $P$.} Suppose that $\vert A \vert = \sum_{j=1}^n \vert \lambda_j \vert$ with $\lambda_j$ being the eigenvalues of $A$, the product of $\vert A \vert$ with the volume of $\sigma$ can be interpreted as the cost of building such an elastic structure whose shape is $\sigma$ and whose elastic constant is summarized by $\vert A \vert$. For example, given a spring with spring constant $k>0$ and length $L$, its mass is equal to $k \cdot L$. The Plateau problem is to show the existence of a minimal polyhedral chain given a boundary condtion $Q$. Here let's solve this question in two difference contexts.

\subsection{Flat Chain Complex}
Suppose $P$ is a stressed $k$-chain, then one can define the {\it flat norm} using Whitney's notation:
\[\mathcal{F}(P)= \inf\{\mathbb{M}(R) + \mathbb{M}(Q) \ \vert \ P = R + \partial T, \text{ where $T$ is a stressed $(k+1)$-chain}\}.\]
Now for each compact subset $K$ of $\mathbb{R}^n$, let $\mathcal{P}_k(K) \subset \mathcal{P}_k(\mathbb{R}^n)$ be the $k$-dimensional stressed chains supported in $K$. Then one complete this metric space with respect to the flat norm $\mathcal{F}$, calling it $\mathcal{F}_k(K)$. The union of these gives the normed abelian group of {\it flat chains}
\[\mathcal{F}_k(\mathbb{R}^n) = \bigcup_{K \text{ compact}} \mathcal{F}_k(K).\]
Applying the compactness theorem for flat chains, the following theorem shows that there exists a minimal flat chain complex given the boundary condition $Q$.

\begin{theorem}\label{ycc:mincostK}
Suppose $Q$ is a stressed polyhedral $(k-1)$-chain such that
\[Q = \sum_{i=1}^m A_i \otimes \sigma_i = S + F,\]
where the $A_i$ are generalized Cauchy stress tensor for each $\sigma_i$, $S$ consists of the stresses, and $F$ consists of the forces. Moreover, the net force and torque in $F$ are both zero. 

Let $K$ be a sufficiently large compact set in $\mathbb{R}^n$ containing the support of $Q$. Then 
there exists a flat $k$-chain $P_0$ in $\mathcal{F}_k(K; Sym_n)$ such that 
\begin{eqnarray*}
\mathbb{M}(P_0) = \inf \left\{\mathbb{M}(P):  P \in \mathcal{P}_k (K; Sym_n), \partial P = Q, P \text{ is structural.}\right\},
\end{eqnarray*}
and
\[\partial P_0 = Q.\]
\end{theorem}

\begin{proof}
Given that $K$ is sufficiently large, the set is at least not empty and so the infimum exists. Denote the infimum by $\alpha$. Let $\{P_{n}\}$ be a minimizing sequence in the set such that $\mathbb{M}(P_n)\rightarrow \alpha $ as $i \rightarrow \infty$.  

Since closed balls are compact in the normed abelian group $(Sym_n, \vert \cdot \vert)$, the compactness theorem for flat chains says that for any $\lambda > 0$, the following set 
\[\{P \in \mathcal{F}_k(K; Sym_n): \mathbb{M}(P) + \mathbb{M}(\partial P) \leq \lambda\}\] 
is compact in the flat norm. Since $M(P_n)$ is bounded above and $M(\partial P_i) = M(Q)$ is finite, it follows that there is a convergent subsequence $\{P_{n_j}\} \rightarrow P_0$ for some flat chain $P_0$ supported in $K$ with respect to the flat norm. Given the fact that mass is lower semicontinuous, 
\[\mathbb{M}(P_0) \leq \liminf_{j \rightarrow \infty} \mathbb{M}(P_{n_j}) = \alpha.\]

Furthermore, this is an equality. Suppose not, then for each $j$, one can find a $k$-polygon $S_j$ and a $(k+1)$-polygon $T_j$, supported in $K$, such that 
\[\mathbb{M}(S_j) + \mathbb{M}(T_j) \leq \mathcal{F}(P_{n_j} + 2^{-j}.\]
As $j \rightarrow \infty$, 
\[\mathcal{F}(P_{n_j}) + 2^{-j} \rightarrow \mathcal{F}(P_0).\]
By definition of the flat norm, $\mathcal{F}(P_0) \leq \mathbb{M}(P_0) < \alpha$. Moreover, 
\[\partial(S_j) = \partial(P_{n_l} - \partial T_j) = Q - 0 = Q,\]
thus a contradiction.

Finally, 
\[\mathcal{F}(\partial P_{n_j} - \partial P_0) \leq \mathcal{F}(P_{n_j} - P_0) \rightarrow 0\]
implies that 
\[\partial P_0 = Q. \]
\end{proof}

\noindent \textbf{Remark} We solve the Plateau problem within closed and bounded subsets of $\mathbb{R}^n$. Letting $K$ become bigger, we have a minimizing sequence in $\mathbb{R}^n$.

\newpage

\subsection{Current}

Consider the vector space of $\mathbb{R}^n$-valued $k$-covectors: 
\[\mathbb{R}^n \otimes \wedge^k \mathbb{R}^n \cong [\wedge^k \mathbb{R}^n]^n. \]
It is generated by elements of the form $v \otimes \omega$ where $v \in \mathbb{R}^n$ and $\omega \in \wedge^k \mathbb{R}^n$. In fact, every element of $[\wedge^k \mathbb{R}^n]^n$ can be expressed in the form
\[(\omega_1, \ldots, \omega_n),\]
where $\omega_j \in \wedge^k\mathbb{R}^n$. For example, 
\[v \otimes \omega = (v^1 \omega, \ldots, v^n \omega), \text{ if } v = (v^1, \ldots, v^n).\]
Since $[\wedge^k \mathbb{R}^n]^n$ is a finite- dimensional real vector space, it is a Hilbert space with an inner product $\langle \cdot, \cdot \rangle$ and an Euclidean norm $\vert \cdot \vert$. 
%It follows that 
%\[\vert (a_1\phi_1, \ldots, a_n \phi_n) \vert = \sqrt{\sum_{j=1}^n a_j^2 \vert \phi_j \vert^2}.\]

Consider now the vector space of $\mathbb{R}^n$-valued, smooth, compactly supported, differential $k$-forms on $\mathbb{R}^n$:
\[\mathbb{R}^n \otimes \mathcal{D}^k(\mathbb{R}^n) \cong [\mathcal{D}^k (\mathbb{R}^n)]^n = C_c^{\infty}(\mathbb{R}^n, \mathbb{R}^n \otimes \wedge^k \mathbb{R}^n)  .\]
It is generated by elements of the form $v \otimes \omega$ where $v \in \mathbb{R}^n$ and $\omega \in \mathcal{D}^k \mathbb{R}^n$. Then one can define its norm as 
\[|v \otimes \omega| = \displaystyle \sup_{x \in \mathbb{R}^n} |v||\omega(x)|. \]
In general, every element in $\mathbb{R}^n \otimes \mathcal{D}^k(\mathbb{R}^n)$ has the form 
\[(\omega_1, \ldots, \omega_n) = \sum_{j=1}^n e_j \otimes \omega_j,\]
where $\omega_j \in \mathcal{D}^k(\mathbb{R}^n)$ and $e_j$ has 1 in the $j$th coordinate. Then one can define its norm as
\[\vert (\omega_1, \ldots, \omega_n) \vert = \displaystyle \sup_{x \in \mathbb{R}^n} \vert (\omega_1(x), \ldots, \omega_n(x)) \vert.\]

For each stressed $k$-chain $A \otimes \sigma$, one may view it as an $\mathbb{R}^n$-valued linear operator on $\mathbb{R}^n \otimes \mathcal{D}^k(\mathbb{R}^n)$ by  
\[A \otimes \sigma (v \otimes \omega) = A(v) \int_{\sigma} \omega.\]

Extending by linearity, every stressed $k$-chain becomes a $\mathbb{R}^n$-valued linear operator on $\mathbb{R}^n \otimes \mathcal{D}^k(\mathbb{R}^n)$. 

\subsubsection{Comass}
Before proceeding to the Plateau problem using currents, let us discuss the comass first. Let's recall that the \textit{comass} of every $k$-covector $\phi \in \wedge^k \mathbb{R}^n$ is defined as
\[\Vert \phi \Vert = \sup\{\phi(X_1 \wedge \cdots \wedge X_k): X_1 \wedge \cdots \wedge X_n \in \wedge_k \mathbb{R}^n, \vert X_1 \wedge \cdots \wedge X_k \vert \leq 1\}.  \]
Comparing to the Euclidean norm, it satisfies the following relationship:
\[\vert \phi \vert \leq \Vert \phi \Vert \leq {n \choose k}^{\frac{1}{2}} \vert \phi \vert.\]
Thus, the {\it comass} of a smooth, compactly supported, differential $k$-form $\omega$ is defined as follows:
\[ \mathbf{M}(\omega) = \displaystyle\sup_{x\in \mathbb{R}^n} \Vert \omega(x) \Vert.\]
One can check that for any oriented $n$-simplex $\sigma$:
\[\sup \left\{\left \vert \int_{\sigma} \omega \right \vert : \mathbf{M}(\omega) \leq 1 \right\} = \text{ vol}(\sigma).\]

Next, let us generalize the definition to vector-valued differential forms. Namely, for every $\Omega \in \mathbb{R}^n \otimes \mathcal{D}^k(\mathbb{R}^n)$, we say that its {\it comass} norm
is defined as:
\[\mathbf{M}(\Omega) = \displaystyle \sqrt{\sum_{j=1}^n \left(\mathbf{M}(\omega_j)\right)^2}, \text{ if } \Omega = (\omega_1, \ldots, \omega_n).\]
Then the \textbf{current mass} of a $k$-dimensional current $T = (T_1, \ldots, T_n)$ can be generalized to
\[\mathbf{M}(T)=\sup\{|T(\Omega)|: \mathbf{M}(\Omega) \leq 1\}.\]
The lemma shows that $\mathbf{M}(T)$ coincides with $\mathbb{M}(T)$.

\begin{lemma}
For every  stressed $k$-chain $T= \sum_{i=1}^m A_i \otimes \sigma_i$ with $A_i$ being symmetric matrices and $\sigma_i$ being oriented $k$-simplices,   
\[\mathbf{M}(T)=\displaystyle\sum_{i=1}^m |A_i|_{op} \cdot \text{vol}(\sigma_i) = \mathbb{M}(T),\] 
where $|A_{op}|$ is the operator norm of the matrix $A$ and it is equal to the maximum of the absolute value of the eigenvalues of $A$.
\end{lemma}

\begin{proof}
On the one hand, if we compute
\[A_i \otimes \sigma_i (\Omega) = A_i(\int_{\sigma_i}\omega_1, \ldots, \int_{\sigma_i}\omega_n),\]
so that 
\[|T(\Omega)| \leq \displaystyle\sum_{i=1}^m |A_i|_{op}\sqrt{\sum_{j=1}^n (\int_{\sigma_i}\omega_j)^2} \leq \displaystyle\sum_{i=1}^m |A_i|_{op}\text{vol}(\sigma_i)\sqrt{\sum_{j=1}^n \mathbf{M}(\omega_j)^2}.\]

On the other hand, if we apply the Riesz representation theorem to $T_i$ to obtain a Radon measure $\nu_i$, then
\[\mathbf{M}(T_i) = |\nu_i|(\mathbb{R}^n).\]
Using this fact, one can show the equality. 

\end{proof}

\subsubsection{Riesz Representation Theorem}

One may solve the Plateau problem after adapting the vector-valued version of the Riesz representation theorem[Thm 4.14, Chp 1, Simon's GMT] to our stressed $k$-chains as follows.
\begin{theorem}
 Suppose $T= \sum_{i=1}^m A_i \otimes \sigma_i$ is a stressed $k$-chain with $A_i$ being symmetric matrices and $\sigma_i$ being oriented $k$-simplices, then
\[T: \mathbb{R}^n \otimes \mathcal{D}^k(\mathbb{R}^n) \longrightarrow \mathbb{R}^n\] 
is linear and for all compact subsets $K \subset \mathbb{R}^n$,
\[\sup\{\vert T(\Omega) \vert: \mathbf{M}(\Omega) \leq 1, \text{ spt }\Omega \subset K \} < \infty. \]
The Riesz Representation Theorem implies that there are positive Radon measures $\nu_i$ and $\nu_i$-measurable functions 
\[\vec{T}_i: \mathbb{R}^n \longrightarrow \mathbb{R}^n \otimes \wedge^k \mathbb{R}^n, \text{ with } \vert \vec{T}_i(x) \vert = 1 \text{ for $\nu_i$ almost all $x$ in $\mathbb{R}^n$},\] 
such that the $i$th entry of $T(v \otimes \omega)$ is equal to 
\[\int_{\mathbb{R}^n} \langle v \otimes \omega(x), \vec{T}_i(x) \rangle d\nu_i,\]
for all $v \in \mathbb{R}^n$ and $\omega \in \mathcal{D}^k(\mathbb{R}^n)$.
\end{theorem}

\begin{proof}
In the statement of the Riesz Representation Theorem, let $X = \mathbb{R}^n$, which is a locally compact Hausdorff space, and let $H = \mathbb{R}^n \otimes \wedge^k \mathbb{R}^n$, which is a finite dimensional real Hilbert space with inner product norm $\vert \cdot \vert$. 
For every $\Omega \in \mathbb{R}^n \otimes \mathcal{D}^k(\mathbb{R}^n)$, it can be viewed as a smooth function from $\mathbb{R}^n$ to $\mathbb{R}^n \otimes \wedge^k \mathbb{R}^n$ with compact support. So $\Omega \in C_c^{\infty}(X, H)$ and  
\[T: C_c^{\infty}(X, H) \longrightarrow \mathbb{R}^n.  \]
Moreover, since every element of $C_c(X, H)$ can be approximated by some $\Omega \in C_c^{\infty}(X, H)$, 
$T$ can be extended linearly to 
\[T: C_c(X, H) \longrightarrow \mathbb{R}^n.\]

It is easy to check that each coordinate function $T_i$ of $T$ is linear. Furthermore, by the previous lemma, 
\[T_i(\Omega) \leq \vert T(\Omega) \vert \leq \mathbf{M}(T),\]
 whenever $\Omega \in C_c^{\infty}(X, H)$ and $\mathbf{M}(\Omega) \leq 1$. It follows that for every compact set $K \subset \mathbb{R}^n$, 
\[\sup\{T_i(\Omega): \Omega \in C_c(X, H), \mathbf{M}(\Omega) \leq 1, \text{ spt} \Omega \subset K\} \leq \mathbf{M}(T) < \infty.\]
Therefore there is a positive Radon measure $\nu_i$ on $X$ and $\nu_i$-measurable function $\vec{T_i}: X \rightarrow H$ with $\vert \vec{T_i} \vert = 1$ for $\nu_i$ almost all $x$ on $X$ such that 
\[T_i(\Omega) = \int_X \langle \Omega(x), \vec{T_i}(x) \rangle d\nu_i(x) \  \text{ for any } \Omega \in C_c(X, H).  \]
In particular, for all $v \in \mathbb{R}^n$ and $\omega \in \mathcal{D}^k(\mathbb{R}^n)$, the $i$th entry of $T(v \otimes \omega)$ is equal to 
\[\int_X \langle v \otimes \omega(x), \vec{T_i}(x) \rangle d\nu_i(x),\]
as desired. 
\end{proof}

\noindent \textbf{Remark} Each $T_i$ is a generalized current, that is a linear operator on the set of $\mathbb{R}^n$-valued smooth differential $k$-forms with compact support. Moreover, every stressed $k$-chain $T = (T_1, \ldots, T_n)$ can be viewed as a vector-valued current.

\begin{corollary}
    Suppose $Q$ is a stressed polyhedral $(k-1)$-chain such that 
    \[Q = \displaystyle\sum_{i}^m A_i \otimes \sigma_i = S+F,\]
where the $A_i$ are generalized Cauchy stress tensors for the $\sigma_i$, $S$ consists of stresses, and $F$ consists of forces. Moreover, the net force and torque in $F$ are both zero. 

Then there exists $T^0 = (T_1^0, \ldots, T_n^0)$ such that 
\[T_i^0 \in C_c(X, H)^{\ast}, \partial T^0 = Q, \] 
\[\mathbb{M}(T^0)=\inf\{\mathbb{M}(T): T \in \mathcal{P}_k(\mathbb{R}^n,Sym_n), \partial T = Q, T \text{ is structural.}\}\]
\end{corollary}

\begin{proof}
According to Banach-Alaoglu, the closed unit ball is compact with respect to the weak$^{\ast}$. If $\{T^n\}$ is minimizing sequence, then there is $T^0$ such that for each $i$, 
\[T_i^0 \in (C_c(X, H))^{\ast},\]
and $T_i^n  \rightarrow T_i^0$ weakly. Namely, for all $\Omega \in \mathbb{R}^n \otimes \mathcal{D}^k(\mathbb{R}^n)$
\[T_i^n (\Omega) \rightarrow T_i^0(\Omega).\]
Then we can apply the Riesz representation theorem to $T_i^0$ to obtain $\nu_i^0$ and $\vec{T_i^0}$. It follows that 
\[\mathbf{M}(T^n) \rightarrow \mathbf{M}(T^0) \Rightarrow \mathbb{M}(T^n)\rightarrow \mathbb{M}(T^0). \]

One can show that the boundary map for currents extends the boundary chain map for polygons. For example, let $A \otimes \sigma$ be a stressed $k$-polygon and $v \otimes \omega$ be a vector-valued differential form of degree $k-1$. On the one hand, as a current
\[(\partial A\otimes \sigma)(v \otimes \omega) = (A\otimes \sigma)(v \otimes  d\omega) = A(v)\int_{\sigma}d\omega.\]
On the other hand, as a chain
\[(\partial A\otimes \sigma)(v \otimes \omega) = (A\otimes \partial \sigma)(v \otimes \omega) = A(v)\int_{\partial \sigma} \omega. \]
They are equal by Strokes' theorem. It follows that for every $\Omega \in \mathbb{R}^n \otimes \mathcal{D}^{k-1}(\mathbb{R}^n)$,
\[\partial T^0(\Omega) = T^0(d\Omega) = \lim_{n\rightarrow\infty}T^n(d\Omega) =\lim_{n\rightarrow\infty}\partial T^n(\Omega) = Q(\Omega). \]
Therefore,
\[\partial T^0 = Q.\]

\end{proof}

 \section{An Important Proposition}
 The next natural question to ask is that :``What does a minimizer look like?" Current research is still going on. We are going to show one progress in discovering the topological properties of minimizers: compressed and stretched springs must be perpendicular to each other at non-boundary points.  
 
\begin{proposition}
Fix a system of pointed forces in equilibrium. Suppose a minimizer exists for this system. Then at any point outside the system, if there are finitely many springs intersecting at the point, the following conditions must hold:
\begin{enumerate}
    \item The sum of boundary forces at this point must be zero.
    \item The angle between a stretched spring and a compressed spring must be $90^{\circ}$.
\end{enumerate}
\end{proposition}

\noindent \textbf{Proof}
    The first condition is obvious, because there was no force in the system at this point originally. So the sum of forces due to these finitely many springs must be zero. 

    For the second condition, we are going to prove by contradiction. Suppose not, one can draw a small circle around the point and reduce the mass further, which is a contradiction to the hypothesis that this is a minimizer. 

    Let's begin with the simplest case of three springs. Without loss of generality, we may assume that there are two stretched springs and one compressed spring.

\noindent \textbf{Case 1} If the angle between each of the stretched springs with the compressed spring is acute, then one can reduce the mass as follows.

\begin{figure}[h]
    \centering
    \includegraphics[width=0.5\linewidth]{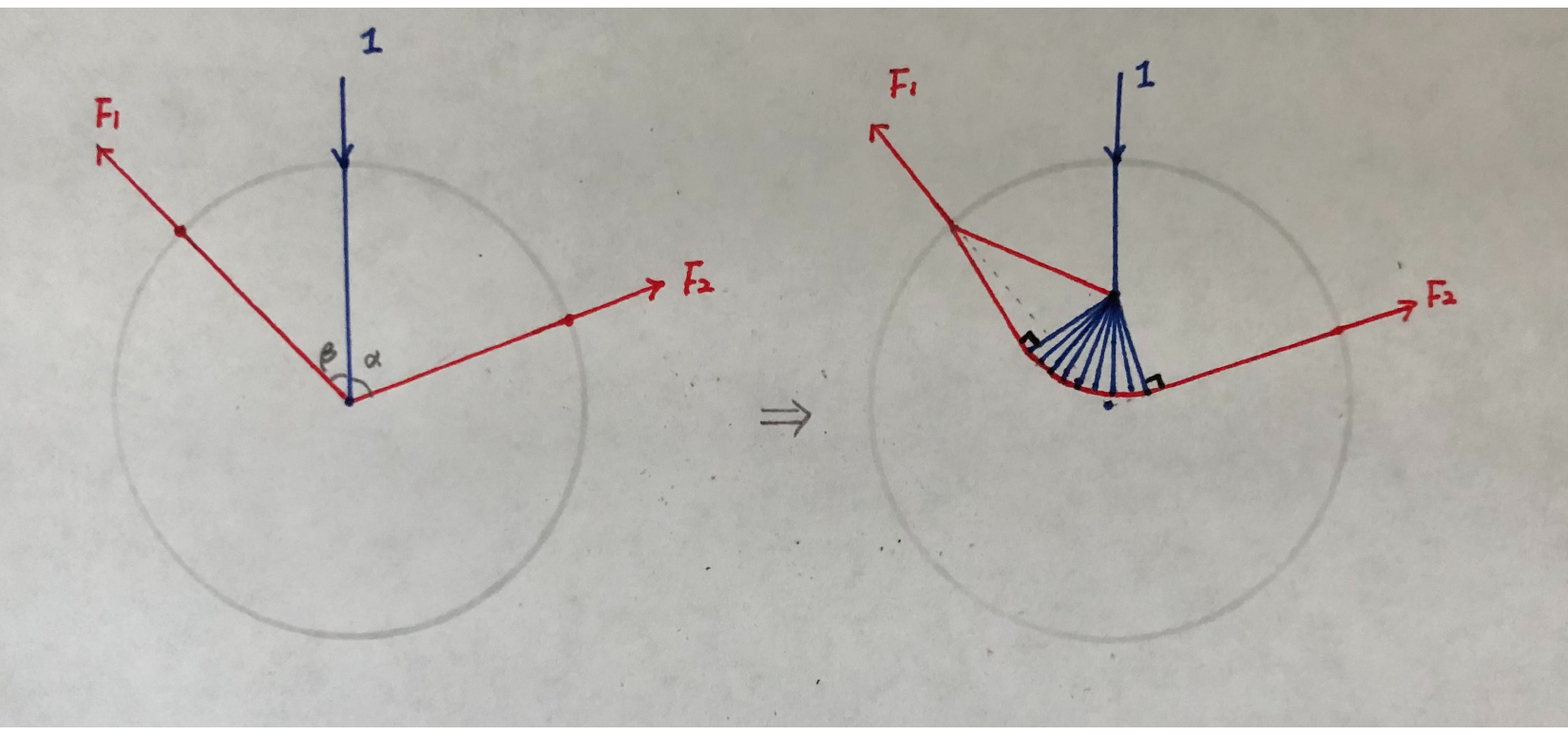}
    \caption{Case 1 of Three Springs}
    \label{fig:case1of3springs}
\end{figure}
    
In figure \ref{fig:case1of3springs}, the blue springs are compressed and the red springs are stretched. After normalization, we may assume that the compressed spring in the middle has unit magnitude and the circle has unit radius as well. Therefore the mass on the left equals $1+|F_1|+|F_2|$. The mass on the right needs some work. Let $x$ be the distance between the bottom of the fan and the center of the circle. One can check that the mass on the right is equal to
\[|F_1|(1-x) + |F_1|(1-x\cos\alpha) +2(\frac{\pi}{2}+\Omega -\alpha)|F_2|\sin(\alpha) x \]
\[+\frac{(|F_1|\sin(\beta)k-|F_1|\cos\beta)(1-2x\cos\beta +x^2)}{k\sin \beta -\cos \beta +x} +|F_2|(\sin \beta \sqrt{1+k^2}-k\sin(\alpha)x). \]
Here, 
\[|F_1| = \frac{\sin \alpha}{\sin(\alpha +\beta)}, |F_2| =\frac{\sin \beta}{\sin(\alpha +\beta)}, \tan(\Omega) = k, \]
\[k = \frac{\sin(\alpha)\sin(\beta)x+(\cos\beta -x)\sqrt{1-2\cos(\beta)x+\cos^2(\alpha)x^2}}{\sin(\beta)\sqrt{1-2\cos(\beta)x+\cos^2(\alpha)x^2}-(\cos\beta-x)\sin(\alpha)x}. \]
The angle $\Omega$ is shown in the figure \ref{fig:case1notation3springs} and the most important thing about it is that $\Omega \rightarrow \beta$ as $x \rightarrow 0$. 

\begin{figure}[h]
    \centering
    \includegraphics[width=0.3\linewidth]{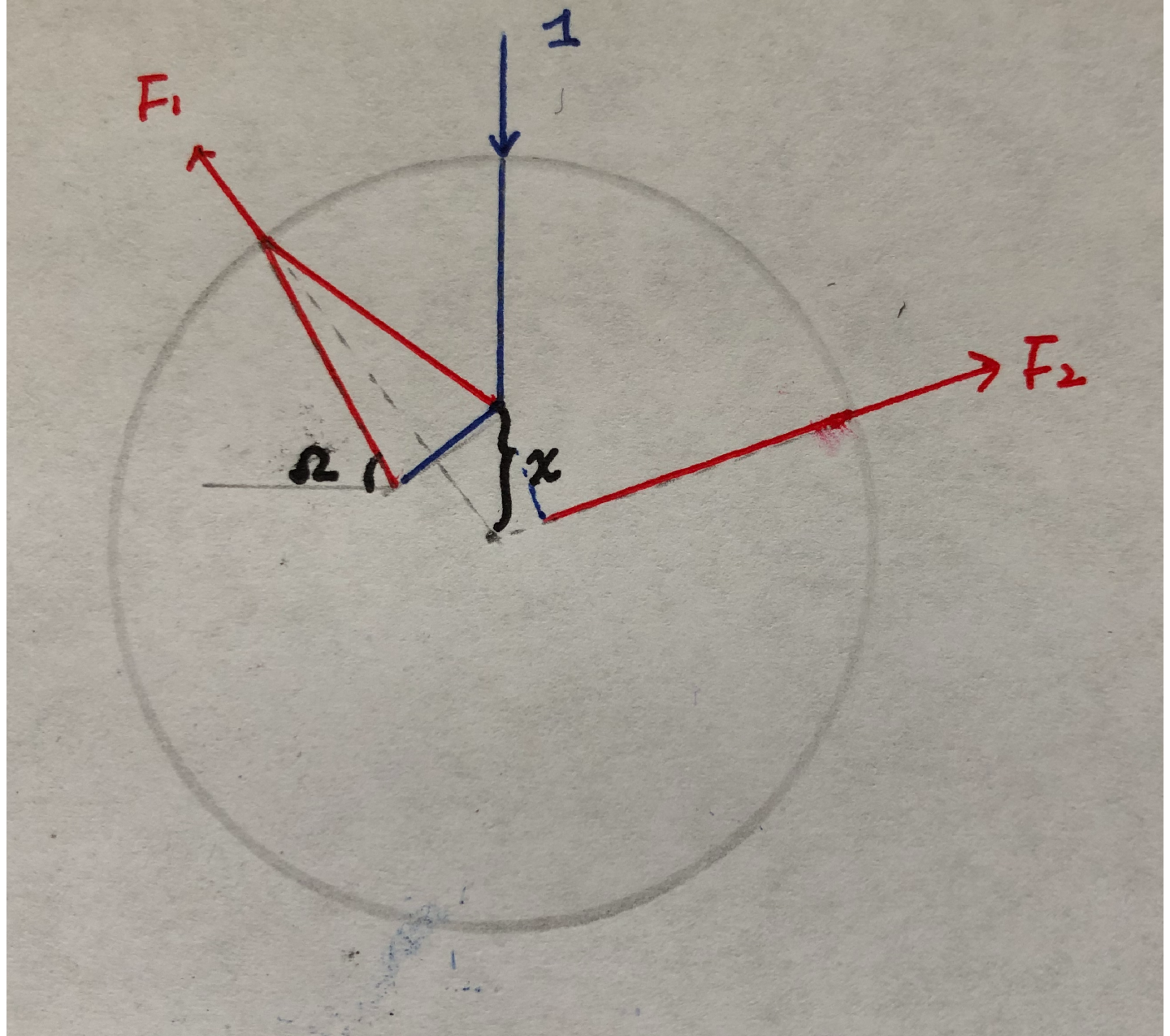}
    \caption{Notation for Case 1 of Three Springs}
    \label{fig:case1notation3springs}
\end{figure}

Let $f(x)$ be the difference between the two masses, and we want to show that $f(x)$ becomes negative as $x$ gets closer to 0 for any arbitrary acute angles $\alpha$ and $\beta$. It suffices to prove that $f'(0)< 0$. One can check that it is equivalent to show 
\[(\pi  -\alpha - \beta)\sin(\alpha)\sin(\beta) < \sin(\alpha + \beta).\]
This indeed holds. 

\noindent \textbf{Case 2} If one of the stretched springs is perpendicular to the compressed spring, say $\alpha = 90^{\circ}$, then the previous argument works verbatim. Furthermore,  the inequality can be simplified to 
\[(\pi/2-\beta)\sin(\beta) < \sin(\pi/2 + \beta).\]

\noindent \textbf{Case 3} If one of the stretched springs makes an obtuse angle with the compressed spring, again the argument is very similar, but there is a slight change in the inequality. Let's replace $\alpha$ with $90^{\circ} + \alpha$ as shown in Figure \ref{fig:case3of3springs}. 

\begin{figure}
    \centering
    \includegraphics[width=0.5\linewidth]{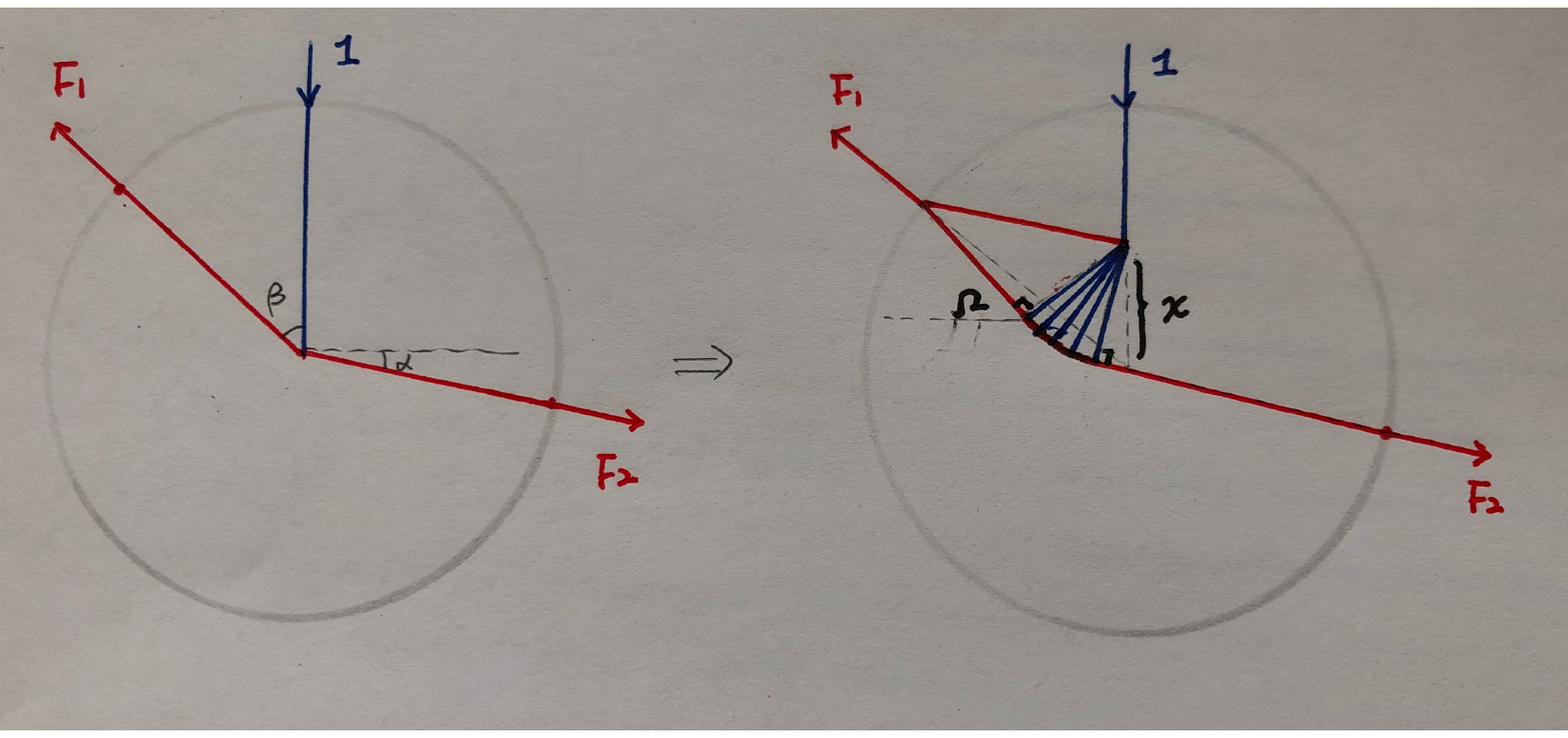}
    \caption{Case 3 of Three Springs}
    \label{fig:case3of3springs}
\end{figure}

One can check the difference $f(x)$ between the new and old systems is equal to
\[-|F_1|-x + |F_2|\sin(\alpha)x + 2|F_2|\cos(\alpha)x(\Omega-\alpha)\]
\[+\frac{(|F_1|\sin(\beta)k-|F_1|\cos\beta)(1-2x\cos\beta +x^2)}{k\sin \beta -\cos \beta +x} +|F_2|(\sin \beta \sqrt{1+k^2}-k\cos(\alpha)x). \]
Here, 
\[|F_1| = \frac{\cos \alpha}{\cos(\alpha +\beta)}, |F_2| =\frac{\sin \beta}{\cos(\alpha +\beta)}, \tan(\Omega) = k, \]
\[k = \frac{\cos(\alpha)\sin(\beta)x+(\cos\beta -x)\sqrt{1-2\cos(\beta)x+\sin^2(\alpha)x^2}}{\sin(\beta)\sqrt{1-2\cos(\beta)x+\sin^2(\alpha)x^2}-(\cos\beta-x)\cos(\alpha)x}. \]
So far everything can be obtained from Case 1 by replacing $\alpha$ with $90^{\circ} + \alpha$. However, when trying to show $f'(0)< 0$, the inequality will be different. 
\[\cos(\alpha) \sin(\beta)(\pi/2-\alpha - \beta) < \cos(\alpha + \beta). \]

In all cases of three springs, we proved that the mass can be reduced in the fan construction as long as $x$ is small enough for any $\alpha$ and $\beta$. There are infinitely many choices for $x$, for example,
\[x = 0.01\cos(\beta)\beta \frac{\cos(\frac{\alpha+\beta}{2})-\sin(\frac{\alpha+\beta}{2})}{\cos(\frac{\alpha-\beta}{2})-\sin(\frac{\alpha-\beta}{2})}.\]

Now let's move onto the next case of four springs. We can find a compressed spring between two stretched springs or vice versa, otherwise the system will not be balanced. 

\noindent \textbf{Case 1} If the angle between the two stretched springs is less than $180^{\circ}$, then we decompose the system into two sub-systems. 

\begin{figure}
    \centering
    \includegraphics[width=0.5\linewidth]{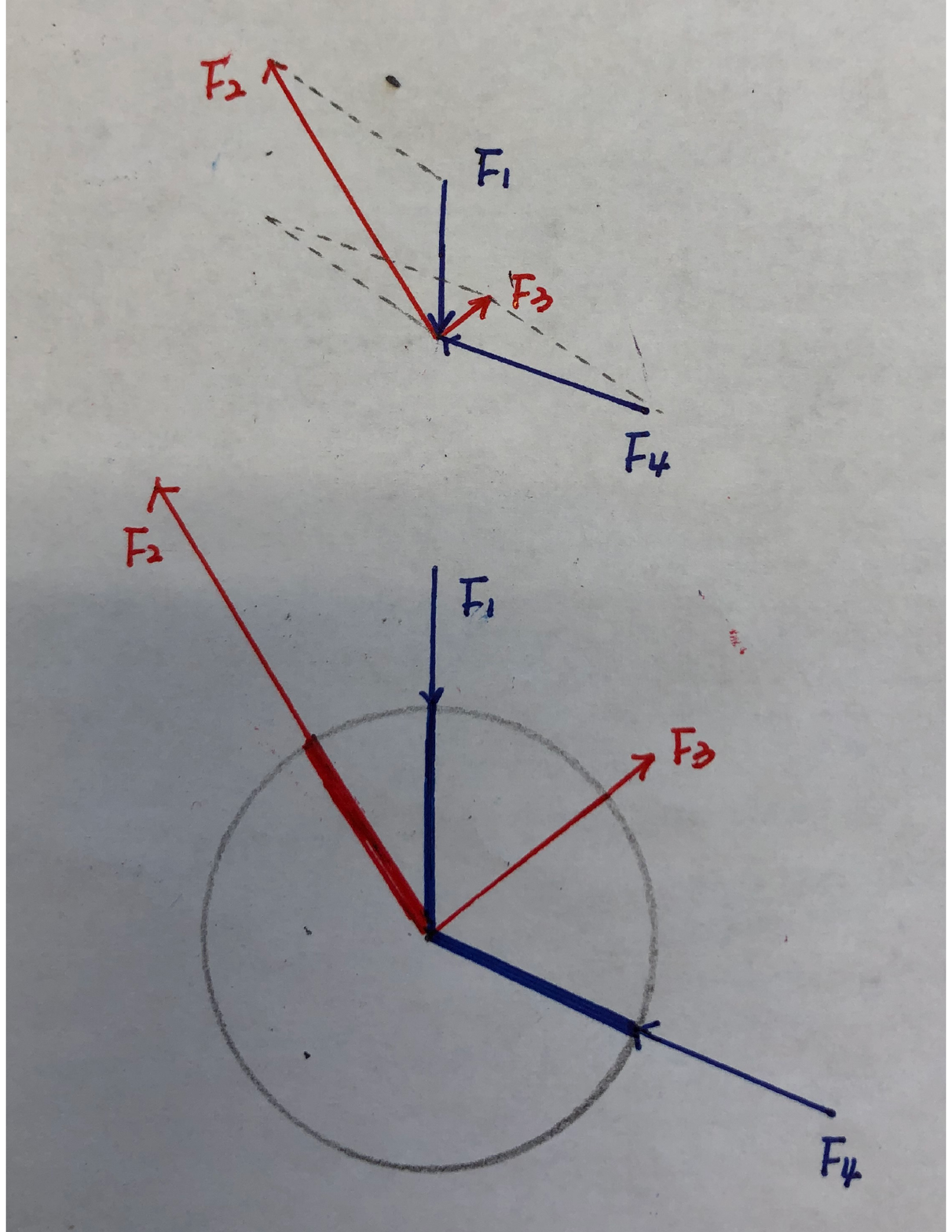}
    \caption{Case 1 of Four Springs}
    \label{fig:case14springs}
\end{figure}

More precisely, in Figure \ref{fig:case14springs}, since the forces $F_1, F_2, F_3$ are linearly independent, one can express $F_1$ as 
\[F_1 = -aF_2 - bF_3, \]
for some positive numbers $a$, $b$. If $c$ is positive and sufficiently close to 0, then $ac < 1$ and $bc < 1$. Since
\[cF_1 + caF_2 + cbF_3 = 0,\]
\[(1-c)F_1 + (1-ac)F_2 + (1-bc)F_3 + F_4 = 0,\]
one can decompose the original system $\{F_1, F_2, F_3, F_4\}$ into $\{cF_1, caF_2, cbF_3\}$ and $\{(1-c)F_1, (1-ac)F_2, (1-bc)F_3, F_4\}$. It remains to reduce the total mass for the system $\{cF_1, caF_2, cbF_3\}$. This can be done based on the previous argument for three springs. 

\noindent \textbf{Case 2} If the angle between the two stretched spring is exactly $180^{\circ}$,  either the two stretched springs are different or the same. See Figure \ref{fig:case24springs} . 

\begin{figure}
    \centering
    \includegraphics[width=0.5\linewidth]{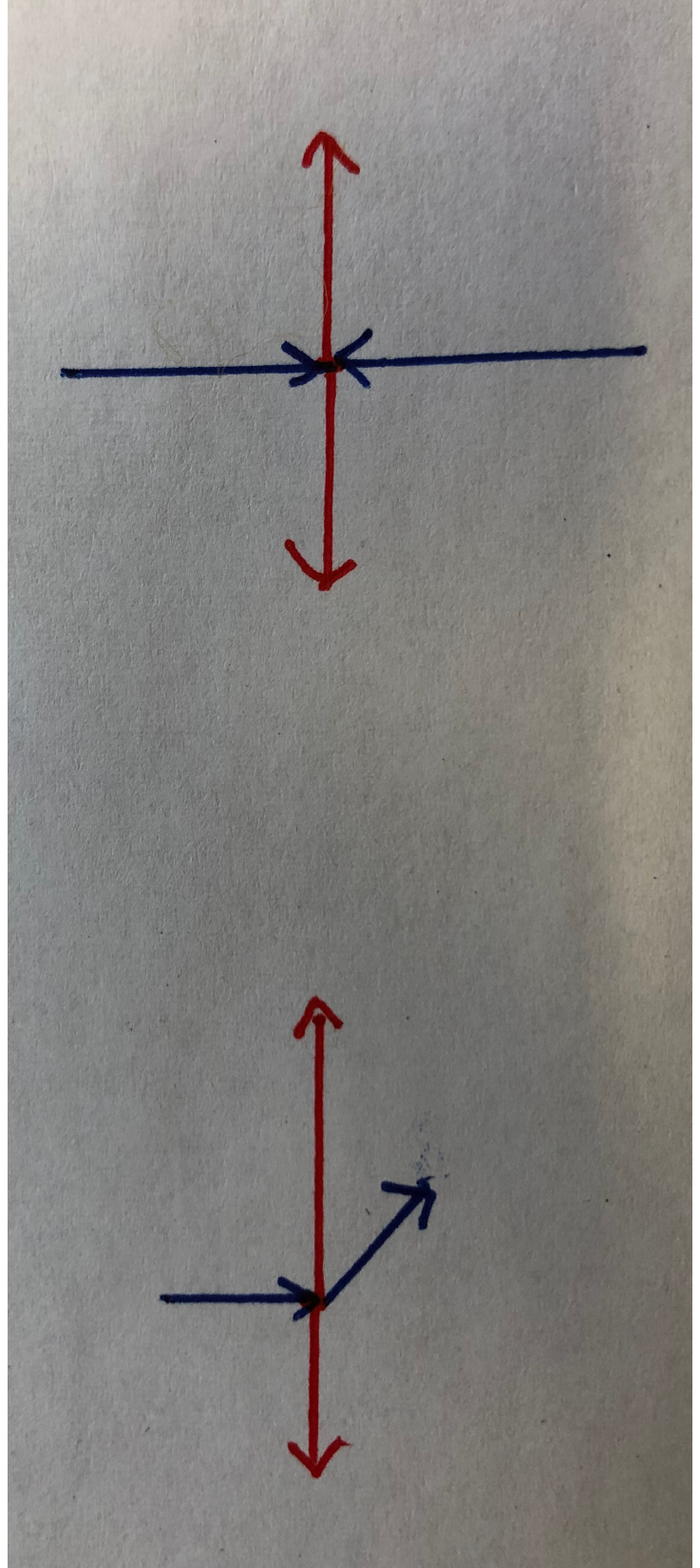}
    \caption{Case 2 of Four Springs}
    \label{fig:case24springs}
\end{figure}

In the first situation,  since the boundary forces due to the stretched springs are not zero, one can find a stretched spring bounded between two compressed springs, therefore returning to Case 1. 
In the second situation,  there is another pair of compressed springs of the same boundary force and forming an angle of $180^{\circ}$. Since we can't apply the same technique as in Case 1, new constructions must be found. Suppose the boundary forces of the compressed springs are $F_1, -F_1$, and of the stretched springs are $F_2, -F_2$. We need an auxillary pointed force $F_3$ satisfying $F_3= F_1 + F_2$. Depending on whether the angle between $F_1$ and $F_3$ is acute, right, or obtuse, again we need to separate into difference cases. 

When the angel is obtuse, Figure \ref{fig:case24springsobtuse}  suggests another way to reduce the mass. 

\begin{figure}
    \centering
    \includegraphics[width=0.5\linewidth]{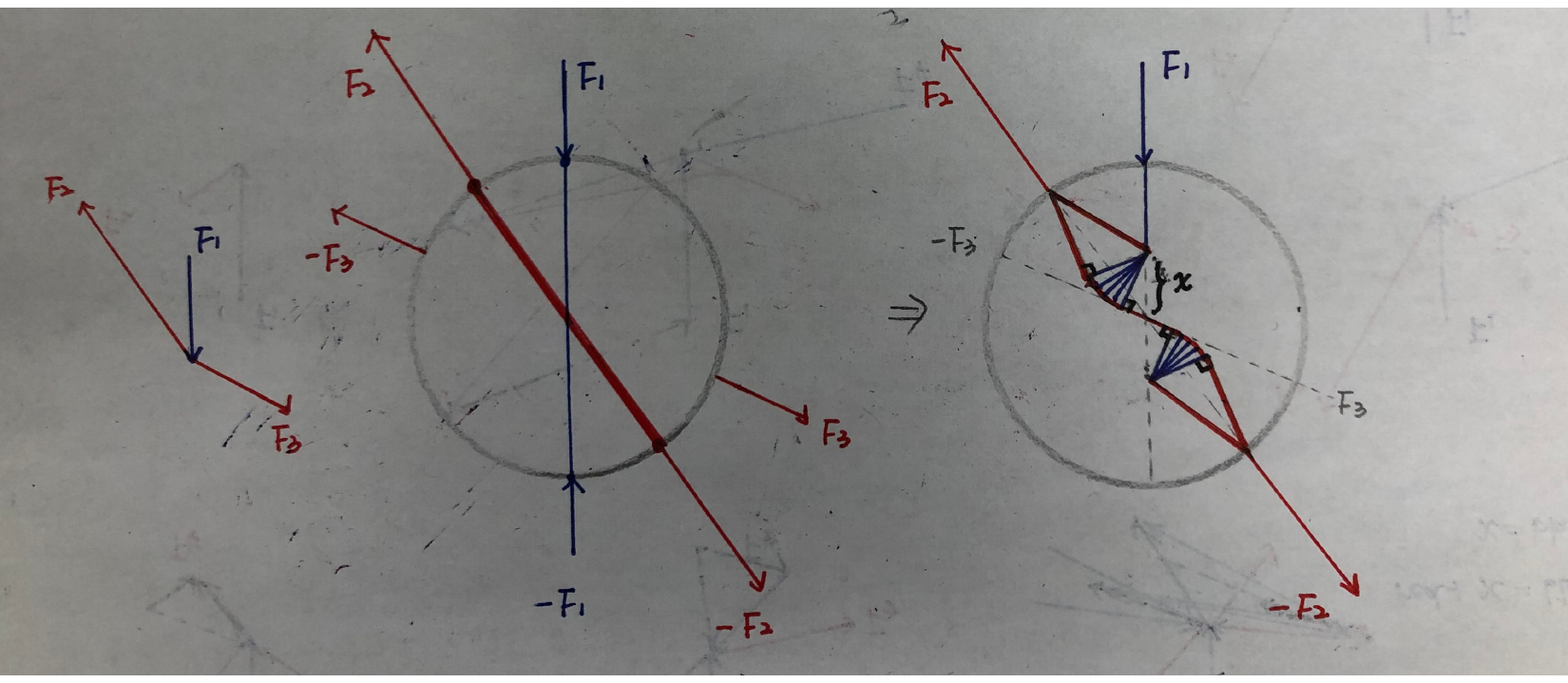}
    \caption{The angle between $F_1$ and $F_3$ is obtuse}
    \label{fig:case24springsobtuse}
\end{figure}

One could apply Case 3 of three springs by sticking two symmetric copies of Figure \ref{fig:case3of3springs} and deleting the stretched springs connecting the center with $F_3$, $-F_3$.

On the other hand, when the angle is acute or right, Figure \ref{fig:case24springsacute} suggests one way to reduce the mass. Although we could not use Figure \ref{fig:case1of3springs} directly, one can make a similar construction. The inequality associated with proving smaller mass as $x$ approaches zero is 
\[(\pi - 2\beta + \cos \beta)\sin \beta < 3 \cos \beta. \]

\begin{figure}
    \centering
    \includegraphics[width=0.5\linewidth]{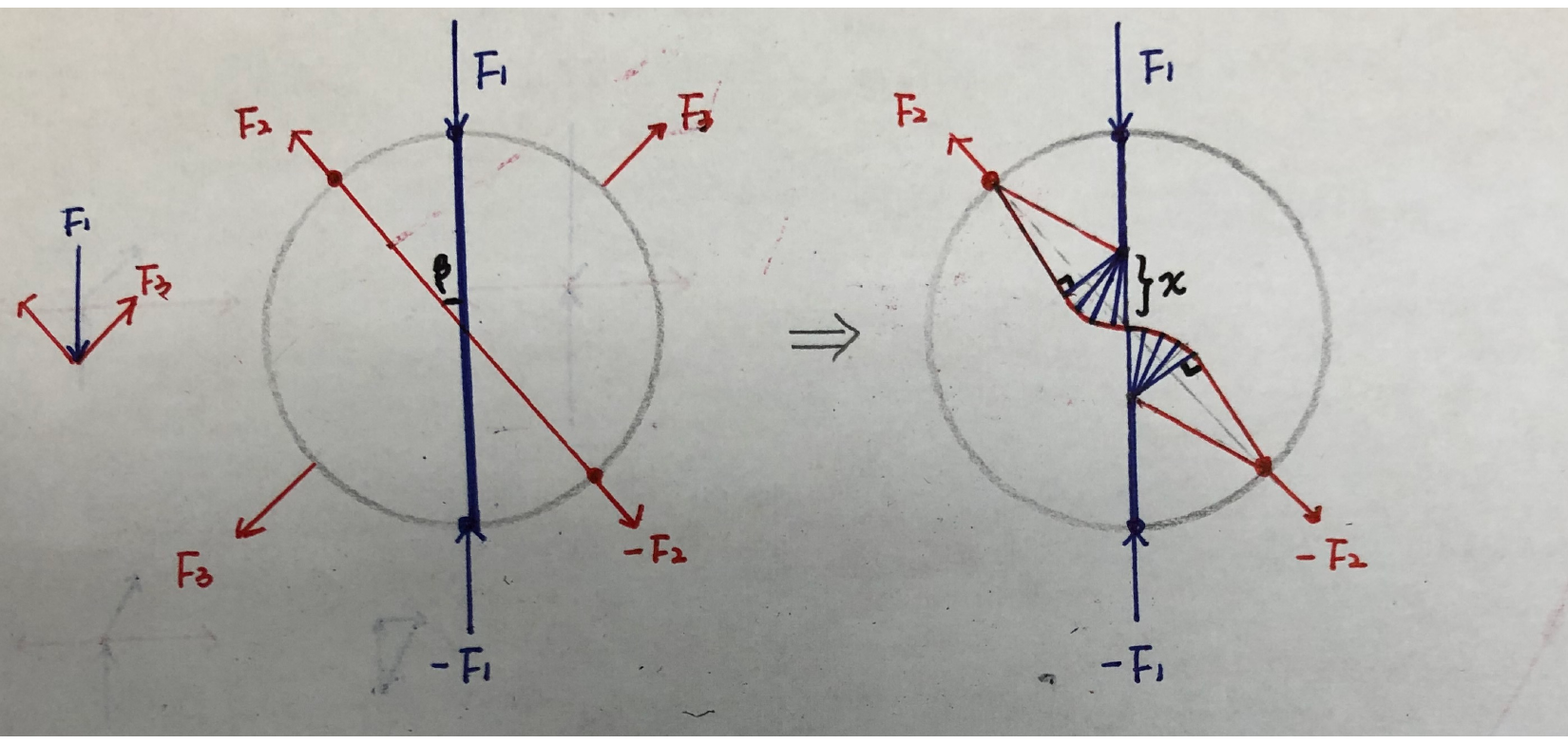}
    \caption{The angle between $F_1$ and $F_3$ is acute or right.}
    \label{fig:case24springsacute}
\end{figure}

Finally let's finish the lemma by looking at the general case. If the number of springs is greater than or equal to five, one can always show that there is one compressed spring between two stretched springs or vice versa. Repeating the same argument as in Case 1 for four springs, one can decompose the system $\{F_1, F_2, F_3, \ldots, F_n\}$,for $n \geq 5$, into the following two sub-systems:
\[cF_1 + caF_2 +cbF_3=0,\]
\[(1-c)F_1 + (1-ac)F_2 + (1-bc)F_3 + F_4 + \cdots + F_n=0.\]

 %Let us call $A \otimes \sigma$ a {\bf generalized current} or a {\bf matrix of currents}. Since $S$ is a countable sum of currents formed in this way, it is called an {\bf integer rectifiable current with matrix coefficients}. 

\section{Conclusion}
In this paper, we looked at the Mitchell Trusses in terms of flat chain complex and measures. Although the Plateau problem can be solved, the proof does not tell us how to find such a minimizer. The important proposition tells us that at a point with no pointed force originally, springs of different types form an angle of $90$ degrees. We conjecture that the `fans' only occur at boundary points.


\section{Bibliography}
\begin{thebibliography}{20}

\bibitem{C}
Y.K. Chen, {\it Matheamtical Results for Michell Trusses}, Rice Thesis, 2022. 

\bibitem{F}
H. Federer, {\it Geometric Measure Theory}, Springer-Verlag New York Inc. 1969.



\bibitem{G1}
W. Gangbo, {\it Michell Truses and Existence of Lines of Principal Actions}, 1960. 

\bibitem{G2}
G. Bouchitt\'{e}, W. Gangbo, and P. Seppecher, {\it Michell Trusses and Existence of Lines of Principal Actions}, Mathematical Models and Methods in Applied Sciences 18.9 (2008), pp.1571-1603


\bibitem{Hemp}
W. S. Hemp, {\it Michell Framework for Uniform Load between Fixed Supports}, Engineering Optimization 1.1 (1974), pp. 61-69. eprint: https://doi.org/10.1080/03052157408960577.

\bibitem{HP}
G. Hegemier and W. Prager, {\it On Michell Trusses}, International Journal on Mechanical Sciences 11.2 (1969), pp. 209-215. 


\bibitem{M}
A. G. Michell, {\it The limits of economy of material in framed-structures}, Philosophical magazine, vol 8 issue 47, 1904. 589-597.

\bibitem{S}
L. Simon, {\it The Introduction to Geometric Measure Theory}, Tsinghua lectures 2016, NTU lectures 2018, Stanford website. 


\bibitem{W}
B. White, {\it Brian White-Topics in GMT (Math 258) Lecture Notes}, Stanford website 2014.

\end{thebibliography}
\end{document}